\documentclass[11pt,twoside,reqno]{amsart}

\usepackage{microtype}
\usepackage[OT1]{fontenc}
\usepackage{type1cm}
\usepackage{amssymb}
\usepackage{enumerate}
\usepackage{xcolor}
\usepackage{mathrsfs}
\usepackage{comment}

\usepackage[left=2.3cm,top=3cm,right=2.3cm]{geometry}
\numberwithin{equation}{section}

\theoremstyle{plain}
\newtheorem{theorem}{Theorem}[section]

\newtheorem{corollary}[theorem]{Corollary}
\newtheorem{proposition}[theorem]{Proposition}

\newtheorem{lemma}[theorem]{Lemma}

\theoremstyle{remark}
\newtheorem{remark}[theorem]{Remark}

\newtheorem{example}[theorem]{Example}
\newtheorem*{ack}{Acknowledgements}

\theoremstyle{definition}

\newcommand{\BB}{\mathcal{B}}
\newcommand{\HH}{\mathcal{H}}

\newcommand{\R}{\mathbb{R}}

\newcommand{\C}{\mathbb{C}}

\newcommand{\N}{\mathbb{N}}
\newcommand{\hhh}{\mathtt{h}}

\newcommand{\iii}{\mathtt{i}}
\newcommand{\jjj}{\mathtt{j}}
\newcommand{\kkk}{\mathtt{k}}
\renewcommand{\lll}{\mathtt{l}}
\newcommand{\mmm}{\mathtt{m}}
\newcommand{\eps}{\varepsilon}
\newcommand{\fii}{\varphi}

\newcommand{\la}{\langle}
\newcommand{\ra}{\rangle}

\renewcommand{\ge}{\geqslant}
\renewcommand{\le}{\leqslant}
\renewcommand{\geq}{\geqslant}
\renewcommand{\leq}{\leqslant}

\DeclareMathOperator{\dimh}{dim_H}

\DeclareMathOperator{\dima}{dim_A}

\DeclareMathOperator{\dist}{dist}
\DeclareMathOperator{\diam}{diam}

\begin{document}

\title{Self-conformal sets with positive Hausdorff measure}

\author{Jasmina Angelevska}
\address[Jasmina Angelevska]
        {Ss. Cyril and Methodius University of Skopje \\
         Faculty of Electrical Engineering and Information Technologies \\
         Rugjer Boshkovikj \\
         1000 Skopje \\
         North Macedonia}
\email{jasmina.angelevska@hotmail.com}

\author{Antti K\"aenm\"aki}
\address[Antti K\"aenm\"aki]
        {Department of Physics and Mathematics \\
         University of Eastern Finland \\
         P.O.\ Box 111 \\
         FI-80101 Joensuu \\
         Finland}
\email{antti.kaenmaki@uef.fi}

\author{Sascha Troscheit}
\address[Sascha Troscheit]
        {Faculty of Mathematics \\
         University of Vienna \\
         Oskar Morgenstern Platz 1, 1090 Wien \\
         Austria}
\email{sascha.troscheit@univie.ac.at}

\date{\today}

\begin{abstract}
  We investigate the Hausdorff measure and content on a class of quasi self-similar sets that
  include, for example, graph-directed and sub self-similar and self-conformal sets. 
  We show that any Hausdorff measurable subset of such a set has comparable Hausdorff
  measure and Hausdorff content. In particular, this proves that graph-directed and sub
  self-conformal sets with positive 
  Hausdorff measure are Ahlfors regular, irrespective of separation conditions. When restricting to
  self-conformal subsets of the real line with Hausdorff
  dimension strictly less than one, we additionally show that the weak separation condition is
  equivalent to Ahlfors regularity and its failure implies full Assouad dimension. In fact, we resolve a self-conformal extension of the
  dimension drop conjecture for self-conformal sets with positive Hausdorff measure by showing that
  its Hausdorff dimension falls below the expected value if and only if there are exact overlaps.
\end{abstract}

\maketitle

\section{Introduction}

Self-conformal sets are a natural generalisation of self-similar sets. Instead of similitudes, they
are defined by using contractive conformal maps $\fii_1,\ldots,\fii_N$. The prime examples include
Julia sets of hyperbolic rational functions on $\C$, such as Julia sets for $z \mapsto z^2+c$ with
$|c| \ge 2.48$. As the only contractive entire functions on $\C$ are the similitudes, one has to
restrict the definition to a bounded open set $\Omega$ where the mappings $\fii_i$ are
contractive. In the real line, the maps $\fii_i$ are contractive $C^{1+\alpha}$-functions with
non-vanishing derivative. The self-conformal set is the unique non-empty compact set $F$ satifying
\begin{equation*}
  F = \bigcup_{i=1}^N \fii_i(F) = \bigcap_{n \in \N} \bigcup_{\iii \in \{1,\ldots,N\}^n} \fii_\iii(X),
\end{equation*}
where $X \subset \Omega$ is any compact set satisfying $\fii_i(X) \subset X$ and $\fii_\iii = \fii_{i_1} \circ \cdots \circ \fii_{i_n}$ for all $\iii = i_1 \cdots i_n$.

We are primarily interested in determining the size of a self-conformal set $F$. If the
``construction pieces'' $\fii_i(X)$ are separated, then, by relying on conformality, one expects the
dimension of $F$ to be close to the value $s$ for which $1 = \sum_{i=1}^N \diam(\fii_i(X))^s \approx
\sum_{i=1}^N \|\fii_i'\|^s$. Intuitively, one should get better and better estimates for the
dimension by iterating this idea. Indeed, this is precisely what happens: it is straightforward to
see that in general, the Hausdorff dimension of $F$ is at most the limiting value of such
approximations, $\dimh(F) \le P^{-1}(0)$, where
\begin{equation*}
  P(s) = \lim_{n \to \infty} \tfrac{1}{n} \log\sum_{\iii \in \{1,\ldots,N\}^n} \|\fii_\iii'\|^s,
\end{equation*}
and if there is enough separation, then $\dimh(F) = P^{-1}(0)$. In fact, Peres, Rams, Simon, and
Solomyak \cite{PeresEtAl2001} have shown that if $s = P^{-1}(0)$, then the $s$-dimensional Hausdorff
measure of $F$ is positive, $\HH^s(F)>0$, if $F$ satisfies the open set condition, a natural
separation condition under which the overlapping of the construction pieces of roughly the same
diameter has bounded multiplicity.

We focus on the case $\dimh(F) < P^{-1}(0)$. At first, it is easy to see that this occurs when
there are exact overlaps, meaning that there are $\iii \ne \jjj$ for which $\fii_\iii|_F =
\fii_\jjj|_F$. A related separation condition is the weak separation condition which, roughly
speaking, is otherwise the same as the open set condition but allows exact overlapping. The famous
dimension drop conjecture claims that exact overlapping is the only way to drop the Hausdorff
dimension of $F$ below $P^{-1}(0)$. Hochman \cite{Hochman2014} has verified the conjecture for all
self-similar sets in the real line defined by algebraic parameters. It should be remarked that
Hochman's proof does not generalise to the self-conformal case.

In the self-similar case, Zerner \cite{Zerner1996} introduced the identity limit criterion,
\begin{equation*}
  \{\fii_\iii^{-1} \circ \fii_\jjj\}_{\iii,\jjj} \text{ does not accumulate to the identity},
\end{equation*}
and showed that it is equivalent to the weak separation condition. The self-conformal case is more
complicated since we cannot use inverses. Nevertheless, in Section \ref{sec:self-conformal}, we
introduce the identity limit criterion for the conformal setting and in our main technical lemma,
Lemma \ref{thm:close-maps}, we show that if it is not satisfied, then there are arbitrary small
$\delta>0$ such that, for some distinct maps $\fii_\iii$ and $\fii_\jjj$,
\begin{equation*}
  |\fii_\iii(x)-\fii_\jjj(x)| \approx \delta\|\fii_\iii'\| \approx \delta\|\fii_\jjj'\|
\end{equation*}
for all $x$. The lemma thus gives the existence of maps which are arbitrarily close to each other in
the relative scale. Applying this observation inductively, we infer that the overlapping of the
construction pieces of roughly the same diameter has unbounded multiplicity and hence, the weak
separation condition does not hold. Conversely, pigeonholing such unbounded multiplicity implies the
existence of two maps being arbitrarily close to each other in the relative scale. Therefore, we see that
the identity limit criterion is equivalent to the weak separation condition also in the
self-conformal case. This is stated in Theorem \ref{thm:wsc-ilc}.

The role of the identity limit criterion is essential in our considerations. The Assouad dimension
of $F$, $\dima(F)$, is the maximal Hausdorff dimension of its weak tangents, the Hausdorff limits of
successive magnifications. In general, the Assouad dimension serves as an upper bound for the
Hausdorff dimension but if the set is Ahlfors regular, then the two dimensions agree. Fraser,
Henderson, Olson, and Robinson \cite{FraserHendersonOlsonRobinson2015} showed that if a
self-similar set in the real line does not satisfy the identity limit criterion, then its Assouad
dimension is $1$. In Theorem \ref{thm:not-wsc-assouad-one}, we generalise this observation to the
self-conformal case. To prove this, we again apply Lemma \ref{thm:close-maps} inductively to find
small scales containing as many equally distributed points of $F$ as we wish. This shows that the
unit interval appears as a weak tangent and proves the result.

In our main result, Theorem \ref{thm:mainTheoremQuasi}, we prove that if $s = \dimh(F)$, then the
$s$-dimensional Hausdorff measure and content are equivalent. An almost immediate consequence of
this is that the positivity of the Hausdorff measure is equivalent to the Ahlfors regularity. The
result generalises the corresponding theorem of Farkas and Fraser \cite{FarkasFraser2015} in the
self-similar case. It should be emphasized that their proof does not generalise to the
self-conformal case. With this theorem, we can now address the dimension drop conjecture on
self-conformal sets in the real line having Hausdorff dimension strictly less than $1$. Indeed, Lau,
Ngai, Wang \cite{LauNgaiWang2009} have shown that the weak separation condition implies $\HH^s(F)>0$
for $s = \dimh(F)$. As mentioned above, this implies Ahlfors regularity and therefore, also the
Assouad dimension is strictly less than $1$. Since this further implies the identity limit criterion
and hence also the weak separation condition, we conclude that all of these conditions are
equivalent. As the only difference between the open set condition and the weak separation condition
is the exact overlapping, we see, by recalling the result of Peres, Rams, Simon, and Solomyak
\cite{PeresEtAl2001}, that the dimension drop conjecture holds for self-conformal sets with positive
Hausdorff measure. This can be considered to be the main consequence of our considerations.
The result is stated in Theorem \ref{thm:main3}.

The rest of the article is organized as follows. We show the equivalence of the Hausdorff measure
and content in a slightly more general setting of quasi self-similar sets in Section
\ref{sec:quasi}. Section \ref{sec:self-conformal} is devoted to the study of self-conformal sets and
their separation conditions in $\R^d$. Results in the real line and dimension drop conjecture are
explored in Section \ref{sec:dim-drop}. The proofs can be found in Sections
\ref{sec:proofMain}--\ref{sec:not-wsc-assouad-one}.

\section{Quasi self-similar sets} \label{sec:quasi}

Recall that the \emph{$s$-dimensional Hausdorff measure} $\HH^s$ of a set $A \subset \R^d$ is defined by
\begin{equation*}
  \HH^s(A) = \lim_{\delta \downarrow 0} \HH^s_\delta(A) = \sup_{\delta>0} \HH^s_\delta(A),
\end{equation*}
where
\begin{equation*}
  \HH^s_\delta(A) = \inf\biggl\{ \sum_i \diam(U_i)^s : A \subset \bigcup_i U_i \text{ and } \diam(U_i) \le \delta \biggr\}
\end{equation*}
is the \emph{$s$-dimensional Hausdorff $\delta$-content} of $A$. The Hausdorff measure is Borel
regular and the Hausdorff content is an outer measure -- usually highly non-additive and not a Borel
measure. However, the Hausdorff content is slightly easier to compute, and is always finite for bounded
sets, irrespective of $s$. 
It is straightforward to see that $\HH^s(A)=0$ if and only if $\HH^s_\infty(A)=0$ and so the Hausdorff measure and
content share the same critical exponent, the \emph{Hausdorff dimension} $\dimh$ of $A$ which is defined by 
$\dimh(A) = \inf\{s:\HH^s(A)=0\}$.

Our main result is the following theorem. We postpone its proof until Section \ref{sec:proofMain}.

\begin{theorem} \label{thm:mainTheoremQuasi}
  Let $F \subset \R^d$ be a non-empty compact set and $s=\dimh(F)$. 
  Suppose that there is a constant $D \ge 1$ such that for each $x \in F$ and $0<r\le \diam(F)$ there
  exists a mapping $g \colon F \to F \cap B(x,r)$ for which 
  \begin{equation} \label{eq:biLipschitz}
    D^{-1}r|y-z| \le |g(y)-g(z)| \le Dr|y-z|
  \end{equation}
  for all $y,z \in F$. 
  Then there exists a constant $C \ge 1$ such that
  \begin{equation*}
    \HH^s(F \cap B(x,r)) \le Cr^s
  \end{equation*}
  for all $x \in \R^d$ and $r>0$, and
  \begin{equation*}
    \HH^s(F\cap A) \le C\HH^s_\infty(F \cap A)
  \end{equation*}
  for all $A \subset \R^d$.
\end{theorem}

Observe that the assumptions of Theorem~\ref{thm:mainTheoremQuasi} are stronger than
those that define quasi self-similar sets; see \cite{Falconer1989} and \cite[\S
3.1]{Falconer1997}.
Quasi self-similar sets differ to the sets we consider by only requiring the lower bound in 
\eqref{eq:biLipschitz} to hold. The upper bound is crucial in \eqref{eq:copy-of-B-diam} and it
seems unlikely that our assumptions are satisfied by quasi self-similarity alone.

The following result is a straightforward corollary of Theorem \ref{thm:mainTheoremQuasi}.
We say
that a set $A \subset \R^d$ is \emph{Ahlfors $s$-regular} if there exists a Radon measure $\mu$ supported on $A$ and a constant $C \ge 1$ such that
\begin{equation} \label{eq:ahlfors}
  C^{-1}r^s \le \mu(B(x,r)) \le Cr^s
\end{equation}
for all $x \in A$ and $0<r<\diam(A)$. 

\begin{proposition}\label{thm:AhlforsQuasi}
  Let $F \subset \R^d$ be a set satisfying the assumptions of Theorem~\ref{thm:mainTheoremQuasi}. If
  $s=\dimh(F)$, then $\HH^s(F)>0$ if and only if $F$ is Ahlfors $s$-regular.
\end{proposition}

\begin{proof}
  Assuming $F$ to be Ahlfors $s$-regular, let $\mu$ be a measure satisfying \eqref{eq:ahlfors}. Since $\mu(F) \le \sum_i \mu(U_i) \le C\sum_i \diam(U_i)^s$ for all $\delta$-covers $\{U_i\}_i$ of $F$, we get $\HH^s_\delta(F) \ge \mu(F) > 0$ for all $\delta > 0$ and, consequently, $\HH^s(F)>0$.
  To show the necessity of the Ahlfors regularity, suppose that $\HH^s(F)>0$. By Theorem
  \ref{thm:mainTheoremQuasi}, there is a constant $C \ge 1$ such that
  \begin{equation} \label{eq:AhlforsQuasi-finite}
    \HH^s|_F(B(x,r)) \le C r^s
  \end{equation}
  for all $x \in F$ and $r>0$. For each $x \in F$ and $0<r<\diam(F)$, let $g_{x,r} \colon F \to F \cap B(x,r)$ be as in \eqref{eq:biLipschitz}. The existence of such mappings implies
  \begin{equation*}
    \HH^s|_F(B(x,r))\geq \HH^s(g_{x,r}(F))\geq D^{-s} \HH^s(F) r^{s}
  \end{equation*}
  for all $x \in F$ and $0<r<\diam(F)$. Recalling that $F$ is compact, it follows from \eqref{eq:AhlforsQuasi-finite} that
  $\HH^s(F)<\infty$ and $\HH^s|_F$ is therefore a Radon measure. We have thus finished the proof.
\end{proof}

\section{Self-conformal sets} \label{sec:self-conformal}

Let $N\geq 2$ and consider the family of $N$ contractions $\{\fii_1,\ldots,\fii_N\}$ on $\R^d$. We
call this family an \emph{iterated function system}. If all the mappings
$\fii_i \colon \R^d\to\R^d$ are strict contractions, then there exists a unique non-empty compact
set $F$, called the \emph{attractor} of the iterated function system, satisfying
\[
  F= \bigcup_{i=1}^N \fii_i(F).
\]
When all the mappings $\fii_i$ are similarities the attractor is known as a \emph{self-similar set}.
In this paper, we consider the larger class of iterated function systems where all the mappings are
conformal contractions and in this case, we refer to $F$ as a \emph{self-conformal set}. 

Let us next give a precise definition for a conformal iterated function system. Fix an open set $\Omega \subset \R^d$.
A $C^1$-mapping $\fii \colon \Omega \to \R^d$ is \emph{conformal} if
the differential $\fii'(x) \colon \R^d \to \R^d$ is a similarity, i.e.\ satisfies $|\fii'(x)y| =
|\fii'(x)||y| \ne 0$ for all $x \in \Omega$ and $y \in \R^d \setminus \{0\}$ and, as a function of $x$, is H\"older continuous, i.e.\ there exist $\alpha,c>0$ such that
\begin{equation} \label{eq:holder-derivatives}
  |\fii'(x) - \fii'(y)| \le c|x-y|^\alpha
\end{equation}
for all $x,y \in \Omega$. For $d \ge 2$, the H\"older continuity follows from the similarity of the differential and injectivity. In fact, conformal mappings in the plane correspond to the holomorphic
functions on $\mathbb{C}$ with non-zero derivative on their respective domain, and in higher
dimensions, by Liouville's theorem, conformal mappings are either homotheties, isometries, or
compositions of reflections and inversions of a sphere. In the one dimensional case, conformal
mappings are simply the $C^{1+\alpha}$-functions with non-vanishing derivative. We say that $\{\fii_i\}_{i=1}^N$ is a \emph{conformal iterated function system} if each $\fii_i$ is an injective conformal mapping on a bounded open
convex set $\Omega$ such that $\overline{\fii_i(\Omega)} \subset \Omega$ and $\|\fii_i'\| := \sup_{x \in \Omega}|\fii_i'(x)| < 1$. There exists a compact set $X \subset \Omega$ such that $\bigcup_{i=1}^N \fii_i(X) \subset X$, which guarantees the existence of the self-conformal set; for details, see Lemma \ref{thm:semiconformal}. Self-conformal sets are a natural generalisation of self-similar sets.

In Section \ref{sec:proof2}, we shall verify that self-conformal sets satisfy the assumptions of
Theorem~\ref{thm:mainTheoremQuasi}. We thus obtain the following result as an immediate corollary of
Theorem \ref{thm:mainTheoremQuasi} and Proposition \ref{thm:AhlforsQuasi}.

\begin{theorem} \label{thm:combinedConformal}
  Let $F \subset \R^d$ be a self-conformal set and $s=\dimh(F)$. Then there exists a constant $C \ge 1$ such that
  \begin{equation*}
    \HH^s(F \cap A) \le C\HH^s_\infty(F\cap A)
  \end{equation*}
  for all $A \subset \R^d$. Furthermore, $\HH^s(F)>0$ if and only if $F$ is Ahlfors $s$-regular.
\end{theorem}

The above theorem extends to graph-directed and sub self-conformal sets in a straightforward manner; see
Remark \ref{rem:graph-directed}. It is pointed out in \cite[\S 4]{FarkasFraser2015} that the
constant $C$ above cannot be chosen to be $1$. We may thus consider that the theorem generalises the
results of Farkas and Fraser
\cite[Theorem 2.1 and Corollary 3.1]{FarkasFraser2015} on graph-directed self-similar sets; see also
\cite[discussion after Proposition 1.11]{Farkas2016}. It is also worthwhile to emphasize that the
method of Farkas and Fraser cannot be applied to prove Theorem \ref{thm:combinedConformal}: their
proof relied on an abstract lemma on measurable hulls which can only be applied if the measure and
content of the whole set are equal.

To exhibit further results, let us introduce more definitions and notation. Let $\{\fii_i\}_{i=1}^N$
be a conformal iterated function system and $F$ be the associated self-conformal set. We use the
convention that whenever we speak about a self-conformal set $F$, then it is automatically
accompanied with a conformal iterated function system which defines it. Let $\Sigma = \{ 1,\ldots,N
\}^\N$ be the collection of all infinite words constructed from integers $\{ 1,\ldots,N \}$. If
$\iii = i_1i_2\cdots \in \Sigma$, then we define $\iii|_n = i_1 \cdots i_n$ for all $n \in \N$. The
empty word $\iii|_0$ is denoted by $\varnothing$. Observe that $\Sigma_* = \bigcup_{n=0}^\infty
\Sigma_n$, where $\Sigma_n = \{\iii|_n : \iii \in \Sigma\}$ for all $n \in \N$, is the free monoid
on $\Sigma_1 = \{ 1,\ldots,N \}$. If $n \in \N$ and $\iii = i_1 \cdots i_n \in \Sigma_n$, then we
write $\fii_\iii = \fii_{i_1} \circ \cdots \circ \fii_{i_n}$. For $\iii \in \Sigma_* \setminus
\{\varnothing\}$ we set $\iii^- = \iii|_{|\iii|-1}$, where $|\iii|$ is the length of $\iii$.

We say that $F$ satisfies the \emph{weak separation condition} if
\begin{equation*}
  \sup \{\#\Phi(x,r) : x \in F \text{ and } r>0 \} < \infty,
\end{equation*}
where
\begin{equation*}
  \Phi(x,r) = \{ \fii_\iii|_F : \diam(\fii_\iii(F)) \le r < \diam(\fii_{\iii^-}(F)) \text{ and } \fii_\iii(F) \cap B(x,r) \ne \emptyset \}
\end{equation*}
for all $x \in \R^d$ and $r>0$. Furthermore, we say that $F$ satisfies the \emph{identity limit criterion} if
\begin{equation*}
  \inf\{ \|\fii_\iii'\|^{-1} \sup_{x \in F}|\fii_\iii(x)-\fii_\jjj(x)| : \iii,\jjj \in \Sigma_* \text{ such that } \fii_\iii|_F \ne \fii_\jjj|_F \} > 0.
\end{equation*}
The weak separation condition for self-conformal sets was introduced by Lau, Ngai, and Wang
\cite{LauNgaiWang2009}. Our definition is strictly weaker than the original one; see Example
\ref{ex:lau-ngai-wang-ex}. This modification was needed to be able to find a definition for the
identity limit criterion equivalent to the weak separation condition.
The following result is proved in Section \ref{sec:wsc-ilc}.

\begin{theorem} \label{thm:wsc-ilc}
  Let $F \subset \R^d$ be a self-conformal set containing at least two points. Then $F$ satisfies
  the weak separation condition if and only if it satisfies the identity limit criterion.
\end{theorem}

The weak separation condition provides us with a sufficient condition for the self-conformal set to
have positive measure. The identity limit criterion gives, at least in principle, a checkable
condition for the positivity.

\begin{proposition} \label{thm:wsc-positive}
  Let $F \subset \R^d$ be a self-conformal set satisfying the weak separation condition and $s=\dimh(F)$. Then $\HH^s(F)>0$.
\end{proposition}

The above result was observed first time by Lau, Ngai, and Wang \cite{LauNgaiWang2009}. Its proof
follows immediately from \cite[Propositions 3.8 and 3.5]{KaenmakiRossi2016}. We remark that
\cite[Proposition 3.8]{KaenmakiRossi2016} uses the original definition of Lau, Ngai, and Wang
\cite{LauNgaiWang2009} (see Example \ref{ex:lau-ngai-wang-ex}) but its proof applies verbatim also
with our definition of weak separation condition.

\section{Dimension drop conjecture} \label{sec:dim-drop}

The \emph{Assouad dimension} of a set $A \subset \R^d$, denoted by $\dima(A)$, is the infimum of all
$s$ satisfying the following: There exists a constant $C \ge 1$ such that each set $A \cap B(x,R)$
can be covered by at most $C(R/r)^s$ balls of radius $r$ centered at $A$ for all $0<r<R$. It is easy
to see that $\dimh(A) \le \dima(A)$ for all sets $A \subset \R^d$.

The proof of the following theorem is postponed until Section \ref{sec:not-wsc-assouad-one}.

\begin{theorem} \label{thm:not-wsc-assouad-one}
  Let $F \subset \R$ be a self-conformal set containing at least two points. If $F$ does not satisfy
  the identity limit criterion, then $\dima(F)=1$.
\end{theorem}

The above result, together with Theorem \ref{thm:wsc-ilc}, generalises the corresponding result of
Fraser, Henderson, Olson, and Robinson \cite[Theorem 3.1]{FraserHendersonOlsonRobinson2015} on
self-similar sets in the real line.

The following corollary generalises the corresponding result of Farkas and Fraser \cite[Corollary 3.2]{FarkasFraser2015} on self-similar sets.

\begin{corollary} \label{thm:main2}
  Let $F \subset \R$ be a self-conformal set containing at least two points such that $s = \dimh(F) < 1$. Then the following five conditions are equivalent:
  \begin{enumerate}
    \item $F$ satisfies the weak separation condition,
    \item $\HH^s(F)>0$,
    \item $F$ is Ahlfors $s$-regular,
    \item $\dima(F) = s$,
    \item $F$ satisfies the identity limit criterion.
  \end{enumerate}
\end{corollary}

\begin{proof}
  The fact that (1) implies (2) follows Proposition \ref{thm:wsc-positive}. Theorem
  \ref{thm:combinedConformal} guarantees that (2) and (3) are equivalent. It is more or less a
  triviality that (3) implies (4); see, for example, \cite[\S 3]{KaenmakiLehrbackVuorinen2013}.
  Finally, Theorems \ref{thm:not-wsc-assouad-one} and \ref{thm:wsc-ilc} show that (4) implies (5)
  and (5) implies (1), respectively.
\end{proof}

A self-conformal set $F$ satisfies the \emph{open set condition} if there exists a non-empty open set
$U \subset \Omega$ such that $\fii_i(U) \subset U$ for all $i$ and $\fii_i(U) \cap \fii_j(U) =
\emptyset$ whenever $i \ne j$. Recall that, by \cite[Corollary 5.8 and Theorem 3.9]{KaenmakiVilppolainen2008}, the open set condition is equivalent to
\begin{equation*} 
  \sup \{\#\Sigma(x,r) : x \in F \text{ and } r>0 \} < \infty,
\end{equation*}
where
\begin{equation*}
  \Sigma(x,r) = \{ \iii \in \Sigma_* : \diam(\fii_\iii(F)) \le r < \diam(\fii_{\iii^-}(F)) \text{ and } \fii_\iii(F) \cap B(x,r) \ne \emptyset \}
\end{equation*}
for all $x \in \R^d$ and $r>0$. Therefore, the open set condition is stronger than the weak separation
condition. The \emph{pressure} $P \colon [0,\infty) \to \R$, defined by
\begin{equation*}
  P(s) = \lim_{n \to \infty} \tfrac{1}{n} \log \sum_{\iii \in \Sigma_n} \|\fii_\iii'\|^s,
\end{equation*}
is well-defined, convex, continuous, and strictly decreasing. In fact, there exists unique $s \ge 0$
for which $P(s)=0$. It is a classical result that if $F$ satisfies the open set condition, then
$\dimh(F)=P^{-1}(0)$; for the latest incarnation of this observation, see \cite[Proposition
3.5]{KaenmakiRossi2016}. 

We say that a self-conformal set $F$ has an \emph{exact overlap} if there exist $\iii, \jjj \in
\Sigma_*$ such that $\iii \ne \jjj$ and $\fii_\iii|_F = \fii_\jjj|_F$. Observe that if $F$ satisfies
the open set condition, then it cannot have exact overlaps. For a self-similar set $F$ in the
real line, according to a folklore ``dimension drop'' conjecture, $\dimh(F)=\min\{1,P^{-1}(0)\}$ or
otherwise 
there is an exact overlap. Hochman \cite[Corollary 1.2]{Hochman2014} has verified the conjecture
under a mild
assumption which is satisfied for example when the associated iterated function system is defined by
algebraic parameters; see \cite[Theorem 1.5]{Hochman2014}. To generalise Hochman's proof for
self-conformal sets in the real line seems difficult since the semigroup generated by $C^{1+\alpha}$
maps is simply too large: there is no invariant metric and dimension $d\in \N$ for which there is a
smooth injection to $\mathbb{R}^d$, which is bi-Lipschitz to its image in any compact neighbourhood of the identity.

However, the following theorem verifies the conjecture for self-conformal sets in the real line
having positive Hausdorff measure. It generalises the corresponding result of Farkas \cite[Corollary 3.13]{FarkasThesis} on self-similar sets.

\begin{theorem} \label{thm:main3}
  Let $F \subset \R$ be a self-conformal set with $\HH^s(F)>0$ for $s = \dimh(F) < 1$. Then $s = P^{-1}(0)$ if and only if there are no exact overlaps.
\end{theorem}

\begin{proof}
  If $s = P^{-1}(0)$, then the assumption that $\HH^s(F)>0$ together with \cite[Theorem
  1.1]{PeresEtAl2001}, implies that $F$ satisfies the open set condition and hence, cannot have
  exact overlaps. If there are no exact overlaps, then, by Corollary \ref{thm:main2}, the assumption
  $\HH^s(F)>0$ implies that $F$ satisfies the weak separation condition. Therefore, by \cite[Theorem
  1.3]{DengNgai2011} (see also \cite[Remark 3.7(2)]{KaenmakiRossi2016}), the lack of exact overlaps
  implies the open set condition and we have $s = P^{-1}(0)$.
\end{proof}

\section{Proof of Theorem \ref{thm:mainTheoremQuasi}} \label{sec:proofMain}

For a bounded set $A \subset \R^d$ we let
\begin{equation*}
  N_r(A) = \min\left\{ k : A \subset \bigcup_{i=1}^k B(x_i,r) \text{ for some } x_1,\ldots,x_k \in
  \R^d \right\}
\end{equation*}
be the least number of balls of radius $r>0$ needed to cover $A$.

\begin{lemma} \label{thm:falconer}
  Let $F \subset \R^d$ be a set satisfying the assumptions of Theorem~\ref{thm:mainTheoremQuasi}. If
  $s=\dimh(F)$, then
  \begin{equation*}
    2^{-s}\HH^s_\infty(F) r^{-s} \le N_r(F) \le D^s r^{-s}
  \end{equation*}
  for all $r>0$ and $\HH^s(F)<(2D)^s$. In particular, $\HH^s(F)>0$ if and only if $0<\HH^s(F)<\infty$.
\end{lemma}

\begin{proof}
  The first claim follows from the definition of $\HH^s_\infty$, the existence of mappings $g \colon F \to F \cap B(x,r)$ satisfying \eqref{eq:biLipschitz}, and \cite[Theorem 3.2]{Falconer1997}. The second claim follows immediately from the first one.
\end{proof}

We are now ready to prove the main theorem.

\begin{proof}[Proof of Theorem \ref{thm:mainTheoremQuasi}]
  We may assume that $\HH^s(F) > 0$ since otherwise there is nothing to prove. This of course
  implies that $\HH^s_\infty(F) > 0$. Write $C = 2 \cdot 2^{4s}D^{3s}\HH^s_\infty(F)^{-1}$. To prove the first claim, suppose,
  for a contradiction, that there exist $x_0 \in \R^d$ and $r_0 > 0$ such that
  \begin{equation} \label{eq:contradiction}
    \HH^s(F \cap B(x_0,r_0)) > C r_0^s.
  \end{equation}
  Fix $n \in \N$ and let $\BB_n$ be a maximal
  collection of pairwise disjoint closed balls of radius $2^{-n}$ centered in $F$. Note that, by
  \cite[Equation (5.4)]{Mattila1995} and Lemma \ref{thm:falconer}, we have
  \begin{equation} \label{eq:BB-count}
    2^{-2s}\HH^s_\infty(F)2^{ns} \le \#\BB_n \le 2^sD^s2^{ns}.
  \end{equation}
  For each $B \in \BB_n$, let $g_B \colon F \to F \cap B$ be as in \eqref{eq:biLipschitz}. 
  It follows that each ball $B$ in the packing $\BB_n$ contains $g_B(F \cap B(x_0,r_0))$, a scaled copy of
  $F \cap B(x_0,r_0)$. Therefore, recalling \eqref{eq:contradiction}, we get
  \begin{equation} \label{eq:copy-of-B-measure}
    \begin{split}
    \HH^s(g_B(F \cap B(x_0,r_0))) &\ge D^{-s}2^{-ns}\HH^s(F \cap B(x_0,r_0))\\
    &> C D^{-s}2^{-ns}r_0^s=
   2 \cdot 2^{4s-ns} D^{2s}\HH^s_\infty(F)^{-1}r_0^s
 \end{split}
  \end{equation}
  for all $B \in \BB_n$.
  Furthermore, since $\diam(g_B(F \cap B(x_0,r_0))) \le D2^{-n}\diam(F \cap B(x_0,r_0)) \leq D
  2^{-n}2r_0 =: \delta_n$, we have
  \begin{equation} \label{eq:copy-of-B-diam}
    \HH^s_{\delta_n }(g_B(F \cap B(x_0,r_0))) = \HH^s_\infty(g_B(F \cap B(x_0,r_0))) \le
    D^s2^{-ns}2^s r_0^s
  \end{equation}
  for all $B \in \BB_n$.

  Now \eqref{eq:copy-of-B-measure} and \eqref{eq:BB-count} imply
  \begin{equation} \label{eq:copy-of-B-measure-sum}
    \sum_{B \in \BB_n} \HH^s(g_B(F \cap B(x_0,r_0))) \ge \#\BB_n 2^{4s-ns+1} D^{2s}\HH^s_\infty(F)^{-1}r_0^s
    \geq 2 \cdot 2^{2s} D^{2s}r_0^s
  \end{equation}
  and \eqref{eq:copy-of-B-diam} and \eqref{eq:BB-count} give
  \begin{equation} \label{eq:copy-of-B-diam-sum}
    \sum_{B \in \BB_n} \HH^s_{\delta_n }(g_B(F \cap B(x_0,r_0))) \le \#\BB_n D^s2^{-ns} 2^s r_0^s
    \leq 2^{2s} D^{2s}r_0^s.
  \end{equation}
  Since, by the fact that the sets $g_B(F \cap B(x_0,r_0))$ are $\HH^s$-measurable and \eqref{eq:copy-of-B-measure-sum},
  \begin{equation*} 
  \begin{split}
    \HH^s(F) &= \HH^s\biggl( F \setminus \bigcup_{B \in \BB_n} g_B(F \cap B(x_0,r_0)) \biggr) +
    \sum_{B \in \BB_n} \HH^s(g_B(F \cap B(x_0,r_0))) \\
    &\ge \HH^s\biggl( F \setminus \bigcup_{B \in \BB_n} g_B(F \cap B(x_0,r_0)) \biggr) +
    2 \cdot 2^{2s}D^{2s} r_0^s
  \end{split}
  \end{equation*}
  and, by \eqref{eq:copy-of-B-diam-sum},
  \begin{equation*} 
  \begin{split}
    \HH^s_{\delta_n}(F) &\le \HH^s_{\delta_n}\biggl( F \setminus \bigcup_{B \in \BB_n} g_B(F \cap B(x_0,r_0)) \biggr) 
    + \sum_{B \in \BB_n} \HH^s_{\delta_n}(g_B(F \cap B(x_0,r_0))) \\
    &\le \HH^s\biggl( F \setminus \bigcup_{B \in \BB_n} g_B(F \cap B(x_0,r_0)) \biggr) +
    2^{2s}D^{2s}r_0^s,
  \end{split}
  \end{equation*}
  we conclude that
  \begin{equation*}
    \HH^s(F) - \HH^s_{\delta_n }(F) \ge
    2 \cdot 2^{2s}D^{2s}r_0^s - 2^{2s}D^{2s}r_0^s = 2^{2s}D^{2s}r_0^s >0.
  \end{equation*}
  This is a contradiction since the lower bound is independent of $n$.

  To show the second claim, let $A\subset \R^d$ and fix $\eps > 0$.
  Choose a countable collection $\{B(x_i,r_i)\}_{i}$ of balls covering $F\cap A$
  such that $\sum_i (2r_i)^s \leq \HH_\infty^s(F\cap A)+\eps$. 
  Applying the first claim, we get
  \begin{equation*}
    \HH^s (F\cap A) \leq \sum_{i} \HH^s(F\cap B(x_i,r_i))
    \leq C\sum_{i}(2r_i)^s \leq C(\HH_\infty^s(F\cap A)+\eps)
  \end{equation*}
  which finishes the proof.
\end{proof}

\section{Proof of Theorem \ref{thm:combinedConformal}} \label{sec:proof2}

The following lemma is standard but we recall it for the convenience of the reader.

\begin{lemma} \label{thm:semiconformal}
  If $\{\fii_i\}_{i=1}^N$ is a conformal iterated function system, then there exists a bounded open convex set $V \subset \R^d$ such that $\fii_i(\overline{V}) \subset V \subset \overline{V} \subset \Omega$ for all $i \in \{1,\ldots,N\}$. Furthermore, if $F \subset V$ is the associated self-conformal set containing at least two points, then there exist a constant $K \ge 1$ such that
  \begin{equation} \label{eq:semiconformal}
    K^{-1}\|\fii_\iii'\||x-y| \le |\fii_\iii(x) - \fii_\iii(y)| \le \|\fii_\iii'\||x-y|
  \end{equation}
  for all $x,y \in V$ and $\iii \in \Sigma_*$,
  \begin{equation} \label{eq:semiconformal-diam}
    \frac{1}{\diam(F)} \diam(\fii_\iii(F)) \le \|\fii_\iii'\| \le \frac{K}{\diam(F)} \diam(\fii_\iii(F))
  \end{equation}
  for all $\iii \in \Sigma_*$, and
  \begin{equation} \label{eq:semiconformal-multi}
    K^{-2}\|\fii_\iii'\|\|\fii_\jjj'\| \le \|\fii_{\iii\jjj}'\| \le \|\fii_\iii'\|\|\fii_\jjj'\|
  \end{equation}
  for all $\iii, \jjj \in \Sigma_*$.
\end{lemma}

\begin{proof}
  Write $d = \dist(\bigcup_{i=1}^N \overline{\fii_i(\Omega)}, \R^d \setminus \Omega)/4 > 0$ and let $U_i$ be the open $d$-neighbourhood of $\fii_i(\Omega)$. It is easy to see that
  \begin{equation} \label{eq:easy-dist}
    \dist(U_i, \R^d \setminus \Omega) \ge 2d
  \end{equation}
  for all $i \in \{1,\ldots,N\}$.
  Indeed, if this was not true, then there are $x \in U_i$ and $w \in \R^d \setminus \Omega$ such that $|x-w|<2d$. As $x \in U_i$, there is $z \in \fii_i(\Omega)$ such that $|z-x|<d$. Therefore, the contradiction $4d \le |z-w| \le |z-x| + |x-w| < 3d$ we obtain proves \eqref{eq:easy-dist}.

  Define $V$ to be the convex hull of $\bigcup_{i=1}^N U_i$. Let us show that
  \begin{equation} \label{eq:conv-dist}
    \dist(V, \R^d \setminus \Omega) \ge 2d.
  \end{equation}
  If this was not the case, then there are $z \in V$ and $w \in \R^d \setminus \Omega$ such that $|z-w| < 2d$. We may assume that $z \not\in \bigcup_{i=1}^N U_i$ since otherwise the contradiction follows immediately from \eqref{eq:easy-dist}. Let $x,y \in \bigcup_{i=1}^N U_i$ be such that $z$ is a convex combination of $x$ and $y$, which we denote by writing $z \in [x,y]$. Let $z'$ be the closest point to $w$ in the line containing the segment $[x,y]$. If $z' \not\in [x,y]$, then there is $v \in \{x,y\}$ such that $v \in [z,z']$. As $v \in \bigcup_{i=1}^N U_i$ and $|v-w| \le |z-w| < 2d$, we get the contradiction again from \eqref{eq:easy-dist}. We may thus assume that $z' \in [x,y] \setminus \bigcup_{i=1}^N U_i$. Notice that $(z'-w) \bot (y-w)$ and $|z'-w| \le |z-w| < 2d$. Let $L_w = \{ w+t(y-x) : t \in \R \}$ be the line parallel to $[x,y]$ going through $w$. Choose $x',y' \in L_w$ so that $(x-x') \bot (y-x)$ and $(y-y') \bot (y-x)$. It follows that $w \in [x',y']$ and $|x-x'| = |y-y'| = |z'-w| < 2d$. By \eqref{eq:easy-dist}, we therefore have $x',y' \in \Omega$. But since $\Omega$ is convex, also $w \in \Omega$ which is a contradiction. Therefore, \eqref{eq:conv-dist} holds and it is thus evident that $\overline{V} \subset \Omega$. Hence, $\fii_i(\overline{V}) \subset \fii_i(\Omega) \subset U_i \subset V$ for all $i \in \{1,\ldots,N\}$.

  By \cite[Lemma 2.2]{MauldinUrbanski1996}, the H\"older continuity of the differentials implies the existence of a constant $K_0 \ge 1$ for which
  \begin{equation} \label{eq:BDP}
    |\fii_\iii'(y)| \le K_0|\fii_\iii'(x)|
  \end{equation}
  for all $x,y \in \Omega$ and $\iii \in \Sigma_*$. Fix $x,y \in \Omega$ and define $x_t =
  (1-t)y+tx$ for all $t \in [0,1]$. Note that, by convexity of $\Omega$, $x_t \in \Omega$ for all $t
  \in [0,1]$. The fundamental theorem of calculus implies that there exists $t_0 \in [0,1]$ such that
  \begin{equation} \label{eq:multivariate}
    |\fii_\iii(x)-\fii_\iii(y)| = \biggl| \int_0^1 \fii_\iii'(x_t) \tfrac{\mathrm{d}}{\mathrm{d}t} x_t \,\mathrm{d}t \biggr| \le |\fii_\iii'(x_{t_0})||x-y|,
  \end{equation}
  which gives the right-hand side inequality in \eqref{eq:semiconformal}. To show the other inequality, fix $x,y \in V$. If $[\fii_\iii(x), \fii_\iii(y)] \cap \partial\fii_\iii(\Omega) \ne \emptyset$, we choose $z \in [\fii_\iii(x), \fii_\iii(y)]$ to be so close to $\partial\fii_\iii(\Omega)$ such that $\fii_\iii^{-1}([\fii_\iii(x),z]) \subset \Omega$ and $|x-\fii_\iii^{-1}(z)| > d$, which is possible by \eqref{eq:conv-dist}. If $[\fii_\iii(x), \fii_\iii(y)] \subset \fii_\iii(\Omega)$, then we write $z = \fii_\iii(y)$. Define $z_t = (1-t)z + t\fii_\iii(x)$ for all $t \in [0,1]$. As above, there exists $t_1 \in [0,1]$ such that
  \begin{equation*}
    |\fii_\iii^{-1}(\fii_\iii(x)) - \fii_\iii^{-1}(z)| = \biggl| \int_0^1 (\fii_\iii^{-1})'(z_t) \tfrac{\mathrm{d}}{\mathrm{d}t} z_t \,\mathrm{d}t \biggr| \le |(\fii_\iii^{-1})'(z_{t_1})||\fii_\iii(x)-z|
  \end{equation*}
  yielding
  \begin{equation} \label{eq:lower-bound-for-6.1}
  \begin{split}
    |\fii_\iii(x) - \fii_\iii(y)| &\ge |\fii_\iii(x) - z| \ge |(\fii_\iii^{-1})'(z_{t_1})|^{-1} |x-\fii_\iii^{-1}(z)| \\ 
    &\ge |(\fii_\iii^{-1})'(z_{t_1})|^{-1} |x-y| \min\biggl\{ 1, \frac{d}{\diam(V)} \biggr\}.
  \end{split}
  \end{equation}
  Note that, by conformality and \eqref{eq:BDP}, $\inf_{w \in \fii_\iii(\Omega)} |(\fii_\iii^{-1})'(w)|^{-1} = \inf_{w \in \Omega} |\fii_\iii'(w)| \ge K_0^{-1}\|\fii_\iii'\|$. Therefore, the left-hand side inequality in \eqref{eq:semiconformal} follows from \eqref{eq:lower-bound-for-6.1} by setting $K = K_0\max\{1,\diam(V)/d\}$. Since both \eqref{eq:semiconformal-diam} and \eqref{eq:semiconformal-multi} follow straightforwardly from \eqref{eq:semiconformal}, we have finished the proof.
\end{proof}

The properties \eqref{eq:semiconformal}--\eqref{eq:semiconformal-multi} are characteristic for
conformal iterated function systems and they are used as a starting point in generalising
self-conformality into metric spaces; see \cite[\S 5]{KaenmakiVilppolainen2008} and \cite[\S
4]{RajalaVilppolainen2010}.

\begin{proof}[Proof of Theorem~\ref{thm:combinedConformal}]
We may clearly assume that $F$ contains at least two points.
Let $x \in F$ and $0<r<\diam(F)$. Pick $\iii \in \Sigma$ such that
$\lim_{n\to\infty}\fii_{\iii\rvert_n}(x_0)=x$ for all $x_0 \in V$ and choose $n \in \N$ for which
$\fii_{\iii|_n}(F) \subset B(x,r)$ but $\fii_{\iii|_{n-1}}(F) \setminus B(x,r) \ne \emptyset$. Note
that the latter property implies $\diam(\fii_{\iii|_{n-1}}(F)) \ge r$. By \eqref{eq:semiconformal}
and \eqref{eq:semiconformal-diam}, we have
  \begin{align*}
    |\fii_{\iii|_n}(y) - \fii_{\iii|_n}(z)| &\ge K^{-2}\|\fii_{\iii|_{n-1}}'\|\min_{i \in \{1,\ldots,N\}}\|\fii_i'\| |y-z| \\
    &\ge \frac{\min_{i \in \{1,\ldots,N\}}\|\fii_i'\|}{K^2\diam(F)} \diam(\fii_{\iii|_{n-1}}(F))|y-z| \ge \frac{\min_{i \in \{1,\ldots,N\}}\|\fii_i'\|}{K^2\diam(F)} r|y-z|
  \end{align*}
  and
  \begin{equation*}
    |\fii_{\iii|_n}(y) - \fii_{\iii|_n}(z)| \le \frac{K}{\diam(F)} \diam(\fii_{\iii|_n}(F))|y-z| \le \frac{2K}{\diam(F)} r|y-z|
  \end{equation*}
  for all $y,z \in F$. 
  By setting
  \begin{equation*}
    D = \max\biggl\{ 1, \frac{K^2 \diam(F)}{\min_{i \in \{1,\ldots,N\}}\|\fii_i'\|}, \frac{2K}{\diam(F)} \biggl\},
  \end{equation*}
  we have thus shown that for each $x \in F$ and $0<r<\diam(F)$ there exist $\iii \in \Sigma$ and $n \in \N$ such that $\fii_{\iii|_n}(F) \subset F \cap B(x,r)$ and
  \begin{equation} \label{eq:quasi-ss}
    D^{-1}r|y-z| \le |\fii_{\iii|_n}(y) - \fii_{\iii|_n}(z)| \le Dr|y-z|
  \end{equation}
  for all $y,z \in F$.
  Theorem~\ref{thm:combinedConformal} follows now immediately from Theorem
  \ref{thm:mainTheoremQuasi} and Proposition \ref{thm:AhlforsQuasi}.
\end{proof}

\begin{remark} \label{rem:graph-directed}
Let $M$ be an $N\times N$-matrix with entries in $\{0,1\}$. We say that a word
$\iii=i_1i_2\cdots\in\Sigma$ is \emph{$M$-admissible} if $M_{i_k,i_{k+1}}=1$ for all $k$. The collection
of $M$-admissible infinite words starting with $i\in\{1,\ldots,N\}$ defines a set when projected onto $\R^d$ by
$\iii \mapsto \lim_{n \to \infty} \fii_{\iii|_n}(x_0)$, where $x_0 \in V$ is fixed, and the
resulting attractor $F_i$ is known as the \emph{graph-directed self-conformal set} of $i$.
It is well-known that if $M$ is irreducible and the iterated function system consists of conformal
contractions that the resulting sets $F_i$ are also quasi self-similar and satisfy $\dimh(F_i) = \dimh(F_{j})$
for all $i,j\in\{1,\ldots,N\}$.
It is also easy to show that there exists $C>0$ such that $\HH^s(F_i)\leq C \HH^s(F_{j})$ for all
$i,j \in \{1,\ldots,N\}$. If $M$ is irreducible, then it is not too difficult to see that
Lemma~\ref{thm:semiconformal} and Theorem~\ref{thm:combinedConformal} hold for graph-directed self-conformal sets.

A \emph{sub self-conformal set} is a non-empty compact set $E\subset F$ which satisfies $E \subset
\bigcup_{i=1}^N \fii_i(E)$, where $F$ is the associated invariant set.
Note that sub self-conformal sets are contained in the invariant set when mapped under $\fii_\iii$,
that is, $\fii_\iii(E)\subset F$.
It is again straightforward
to check that Lemma~\ref{thm:semiconformal} and Theorem~\ref{thm:combinedConformal}
hold for sub self-similar sets.
Generally, the images of graph-directed self-conformal sets are not contained in themselves under
$\fii_\iii$ for all $\iii$ and it is easy to
find examples such that the sets $F_i$ are not sub self-conformal. However, some authors prefer to
define a single graph-directed set using subshifts of finite type. In our notation this amounts to
considering $F=\bigcup_{i=1}^N F_i$. For such $F$ we have $\fii_\iii(F)\subset F$ and
thus $F$ is a sub self-conformal set.

For both cases above we have omitted detailed proofs to avoid cumbersome notation of $M$-admissible
words and arbitrary subsets. 
\end{remark}

\section{Proof of Theorem \ref{thm:wsc-ilc}} \label{sec:wsc-ilc}

The proof of Theorem \ref{thm:wsc-ilc} is split into two parts, Propositions \ref{thm:ILC-implies-WSC} and \ref{thm:WSC-implies-ILC}.

\begin{proposition} \label{thm:ILC-implies-WSC}
  Let $F \subset \R^d$ be a self-conformal set. If $F$ satisfies the identity limit criterion, then
  it satisfies the weak separation condition.
\end{proposition}

\begin{proof}
  We prove that the failure of the weak separation condition implies the failure of the identity
  limit criterion. Our goal, therefore, is to show that for every $\eps>0$ there are $\iii,\jjj \in
  \Sigma_*$ such that
  \begin{equation} \label{eq:ILC-implies-WLC-1}
    0 < \sup_{x \in F} |\fii_\iii(x) - \fii_\jjj(x)| \le \eps\max\{\|\fii_\iii'\|,\|\fii_\jjj'\|\}.
  \end{equation}
  Let $K \ge 1$ be as in Lemma \ref{thm:semiconformal}, fix $\eps>0$, and choose
  \begin{equation} \label{eq:ILC-implies-WLC-2}
    0 < \delta \le \min\biggr\{ \eps \biggl( 4 + \frac{2K^2\diam(F)}{\min_{i \in \{1,\ldots,N\}}\diam(\fii_i(F))} \biggr)^{-1} , \tfrac{1}{2}\diam(F), 1 \biggl\}.
  \end{equation}
  Let $\{B(x_i,\delta)\}_{i=1}^M$ be a maximal collection of pairwise disjoint closed balls centered
  at $F$. Observe that if $\delta \le \tfrac12 \diam(F)$, then $M \le \diam(F)^d \delta^{-d}$.

  Since the weak separation condition does not hold, there exist a point $z \in F$ and a radius $r>0$ such that
  \begin{equation*}
    \#\Phi(z,r) > (5^d \delta^{-d})^M.
  \end{equation*}
  Note that $\fii(F) \subset B(z,2r)$ for all $\fii \in \Phi(z,r)$. Let $\{B_j\}_{j=1}^L$ be a
  minimal cover of $B(z,2r)$ of balls of radius $\delta r$ centered at $B(z,2r)$. Observe that if
  $\delta \le 1$, then $L \le 5^d \delta^{-d}$. Moreover, for each $\fii \in \Phi(z,r)$ there is a
  map $\psi \colon \{1,\ldots,M\} \to \{1,\ldots,L\}$ given by $\psi(i) = j$, where $j \in
  \{1,\ldots,L\}$ such that $\fii(x_i) \in B_j$. Note that there can be at most $L^M$ many different
  maps $\psi$. Since $\# \Phi(z,r) > L^M$, there have to be two maps $\fii_\iii, \fii_\jjj \in \Phi(z,r)$ such that
  \begin{equation} \label{eq:ILC-implies-WLC-3}
    \fii_\iii|_F \ne \fii_\jjj|_F \text{ and for each } i \in \{1,\ldots,M\} \text{ it holds that } \fii_\iii(x_i),\fii_\jjj(x_i) \in B_j
  \end{equation}
  for some $j \in \{1,\ldots,L\}$.

  Let $\iii,\jjj \in \Sigma_*$ satify \eqref{eq:ILC-implies-WLC-3}. Fix $x \in F$ and choose $x_0 \in \{x_i\}_{i=1}^M$ such that $|x-x_0| \le |x-x_i|$ for all $i \in \{1,\ldots,M\}$. Note that, since $\{B(x_i,2\delta)\}_{i=1}^M$ covers $F$, we have $|x-x_0| \le 2\delta$. It follows from the triangle inequality, Lemma \ref{thm:semiconformal}, and \eqref{eq:ILC-implies-WLC-2} that
  \begin{align*}
    |\fii_\iii(x) - \fii_\jjj(x)| &\le |\fii_\iii(x) - \fii_\iii(x_0)| + |\fii_\iii(x_0) - \fii_\jjj(x_0)| + |\fii_\jjj(x_0) - \fii_\jjj(x)| \\
    &\le \|\fii_\iii'\||x-x_0| + 2\delta r + \|\fii_\jjj'\||x_0-x| \\
    &\le 2\delta\max\{\|\fii_\iii'\|,\|\fii_\jjj'\|\}\biggl( 2 + \frac{K^2\diam(F)}{\min_{i \in \{1,\ldots,N\}}\diam(\fii_i(F))} \biggr) \\
    &\le \eps\max\{\|\fii_\iii'\|,\|\fii_\jjj'\|\}.
  \end{align*}
  This proves \eqref{eq:ILC-implies-WLC-1} and finishes the proof.
\end{proof}

Before going into Proposition \ref{thm:WSC-implies-ILC}, we prove three technical lemmas. We say that $F$ is \emph{uniformly perfect} if there exists a constant $H \ge 1$ such that
\begin{equation} \label{eq:unif-perfect}
  F \cap B(x,r) \setminus B(x,r/H) \ne \emptyset
\end{equation}
for all $x \in F$ and $0<r<\diam(F)$.

\begin{lemma} \label{thm:unif-perfect}
  Let $F \subset \R^d$ be a self-conformal set. Then the following three conditions are equivalent:
  \begin{enumerate}
    \item $F$ is uniformly perfect,
    \item $\dimh(F)>0$,
    \item $F$ contains at least two points.
  \end{enumerate}
\end{lemma}

\begin{proof}
  If $F$ is uniformly perfect, then \cite[Corollary 4.2]{JarviVuorinen1996} shows that $\dimh(F)>0$,
  which clearly implies that $F$ contains at least two points. Therefore, it suffices to show that
  (3) implies (1). Let $K \ge 1$ be as in Lemma \ref{thm:semiconformal} and
  \begin{equation*}
    H = \frac{3K^3}{\min_{i \in \{1,\ldots,N\}}\|\fii_i'\|} + 1.
  \end{equation*}
  Let $x \in F$ and $0<r<\diam(F)$. Since $F$ contains at least two points, there exists a point $y
  \in F$ such that $y \ne x$. Let $\iii \in \Sigma$ be such that $\lim_{n\to\infty}\fii_{\iii|_n}(y)
  = x$. Write $d = |x-y| > 0$ and choose $n_0 \in \N$ such that $\diam(\fii_{\iii|_{n_0}}(F)) <
  \tfrac{d}{2}$ and
  \begin{equation*}
    \frac{\tfrac32 dK^2 + K^2\diam(F)(\max_{i \in \{1,\ldots,N\}}\|\fii_i'\|)^{n_0}}{(\tfrac12 dK^{-1} - \diam(F)(\max_{i \in \{1,\ldots,N\}}\|\fii_i'\|)^{n_0})\min_{i \in \{1,\ldots,N\}}\|\fii_i'\|} \le H
  \end{equation*}
  Choose $n \ge n_0$ such that
  \begin{equation} \label{eq:unif-perf-choice-of-r}
    \|\fii_{\iii|_n}'\|(\tfrac32 d + \diam(F)\|\fii_{\iii|_{n_0}}'\|) \le r < \|\fii_{\iii|_{n-1}}'\|(\tfrac32 d + \diam(F)\|\fii_{\iii|_{n_0}}'\|)
  \end{equation}
  and note that it suffices to prove the claim for all $0<r<r_0$, where $0<r_0<\diam(F)$. Let $z \in
  \fii_{\iii|_{n_0}}(F)$ and observe that $x, \fii_{\iii|_n}(z) \in \fii_{\iii|_n\iii|_{n_0}}(F)$.
  Therefore, by \eqref{eq:semiconformal} and \eqref{eq:semiconformal-diam},
  \begin{equation} \label{eq:unif-perf-upper}
  \begin{split}
    |\fii_{\iii|_n}(y)-x| &\le |\fii_{\iii|_n}(y)-\fii_{\iii|_n}(z)| + |\fii_{\iii|_n}(z)-x| \\
    &\le \|\fii_{\iii|_n}'\||y-z| + \diam(F)\|\fii_{\iii|_n\iii|_{n_0}}'\| \\
    &\le \|\fii_{\iii|_n}'\|(\tfrac32 d + \diam(F)\|\fii_{\iii|_{n_0}}'\|) \le r
  \end{split}
  \end{equation}
  and
  \begin{equation} \label{eq:unif-perf-lower-est}
  \begin{split}
    |\fii_{\iii|_n}(y)-x| &\ge |\fii_{\iii|_n}(y)-\fii_{\iii|_n}(z)| - |\fii_{\iii|_n}(z)-x| \\
    &\ge K^{-1}\|\fii_{\iii|_n}'\||y-z| - \diam(F)\|\fii_{\iii|_n\iii|_{n_0}}'\| \\
    &\ge \|\fii_{\iii|_n}'\|(\tfrac12 dK^{-1} - \diam(F)\|\fii_{\iii|_{n_0}}'\|).
  \end{split}
  \end{equation}
  By \eqref{eq:semiconformal-multi}, we have $\|\fii_{\iii|_{n-1}}'\| \le K^2(\min_{i\in\{1,\ldots,N\}}\|\fii_i'\|)^{-1} \|\fii_{\iii|_n}'\|$ and hence, by \eqref{eq:unif-perf-choice-of-r}, the choice of $H \ge 1$, and \eqref{eq:unif-perf-lower-est},
  \begin{equation} \label{eq:unif-perf-lower}
  \begin{split}
    r &< \|\fii_{\iii|_{n}}'\|\frac{\tfrac32 dK^2 + K^4(\min_{i\in\{1,\ldots,N\}}\|\fii_i'\|)^{-1}\diam(F)\|\fii_{\iii|_{n}}'\|}{\min_{i\in\{1,\ldots,N\}}\|\fii_i'\|} \\
    &\le H\|\fii_{\iii|_n}'\|(\tfrac12 dK^{-1} - \diam(F)\|\fii_{\iii|_{n_0}}'\|) \le D|\fii_{\iii|_n}(y)-x|.
  \end{split}
  \end{equation}
  Therefore, by \eqref{eq:unif-perf-upper} and \eqref{eq:unif-perf-lower},
  \begin{equation*}
    \fii_{\iii|_n}(y) \in B(x,r) \setminus B(x,r/H)
  \end{equation*}
  and we conclude that $F$ is uniformly perfect.
\end{proof}

\begin{lemma} \label{thm:derivatives-holder}
  Let $\{\fii_i\}_{i=1}^N$ be a conformal iterated function system. Then there are constants $\alpha,c>0$ such that
  \begin{equation*}
    |\fii_\iii'(x)-\fii_\iii'(y)| \le c\|\fii_\iii'\||x-y|^\alpha
  \end{equation*}
  for all $x,y \in V$ and $\iii \in \Sigma_*$.
\end{lemma}

\begin{proof}
  Let $x,y \in V$ and fix $\iii = i_1\cdots i_n \in \Sigma_n$ for some $n \in \N$. Write
  $\sigma^j(i_1\cdots i_n) = i_{j+1}\cdots i_n$,
  \begin{equation*}
    x_j = \fii_{\sigma^j(\iii)}(x) \quad\text{and}\quad y_j = \fii_{\sigma^j(\iii)}(y),
  \end{equation*}
  and note that, by the chain rule, $\fii_{\iii|_j}'(x_j) = \fii_{i_1}'(x_1) \cdots \fii_{i_j}'(x_j)$ for all $j \in \{1,\ldots,n\}$. We interpret $x_n = x$ and $y_n = y$. Write also
  \begin{equation*}
    d_j = \fii_{i_j}'(x_j) - \fii_{i_j}'(y_j)
  \end{equation*}
  and observe that, by \eqref{eq:holder-derivatives}, there exist constants $\alpha,c>0$ such that
  \begin{equation} \label{eq:dj-estimate}
    |d_j| \le c|x_j-y_j|^\alpha \le c\|\fii_{\sigma^j(\iii)}'\|^\alpha |x-y|^\alpha \le c \Bigl( \max_{i \in \{1,\ldots,N\}}\|\fii_i'\| \Bigr)^{\alpha(n-j)}|x-y|^\alpha
  \end{equation}
  for all $j \in \{1,\ldots,n\}$. Since
  \begin{equation*}
    \fii_{\sigma^{j-1}(\iii)}'(x) - \fii_{\sigma^{j-1}(\iii)}'(y) = \fii_{i_j}'(x_j)\bigl( \fii_{\sigma^j(\iii)}'(x)-\fii_{\sigma^j(\iii)}'(y) \bigr) + d_j\fii_{\sigma^j(\iii)}'(y)
  \end{equation*}
  for all $j \in \{1,\ldots,n\}$, we recursively see that
  \begin{equation} \label{eq:erotus}
  \begin{split}
    \fii_\iii'(x) - \fii_\iii'(y) &= \fii_{i_1}'(x_1)\bigl( \fii_{\sigma(\iii)}'(x)-\fii_{\sigma(\iii)}'(y) \bigr) + d_1\fii_{\sigma(\iii)}'(y) \\
    &= \fii_{\iii|_2}'(x_2)\bigl( \fii_{\sigma^2(\iii)}'(x)-\fii_{\sigma^2(\iii)}'(y) \bigr) + \fii_{i_1}'(x_1)d_2\fii_{\sigma^2(\iii)}'(y) + d_1\fii_{\sigma(\iii)}'(y) \\
    &= \cdots = \sum_{j=1}^n \fii_{\iii|_{j-1}}'(x_{j-1})d_j\fii_{\sigma^j(\iii)}'(y).
  \end{split}
  \end{equation}
  Observe that, by \eqref{eq:semiconformal-multi}, we have $\|\fii_{\iii|_{j-1}}'\|\|\fii_{\sigma^j(\iii)}'\| \le K^2\|\fii_\iii'\|$ for all $j \in \{1,\ldots,n\}$ and hence, by \eqref{eq:erotus} and \eqref{eq:dj-estimate},
  \begin{align*}
    |\fii_\iii'(x) - \fii_\iii'(y)| &\le \sum_{j=1}^n |\fii_{\iii|_{j-1}}'(x_{j-1})| |d_j| |\fii_{\sigma^j(\iii)}'(y)| \le K^2\|\fii_\iii'\| \sum_{j=1}^n |d_j| \\
    &\le \frac{cK^2}{1-\max_{i \in \{1,\ldots,N\}}\|\fii_i'\|^\alpha} \|\fii_\iii'\| |x-y|^\alpha
  \end{align*}
  as claimed.
\end{proof}

\begin{lemma} \label{thm:close-maps}
  Let $\{\fii_i\}_{i=1}^N$ be a conformal iterated function system and $F \subset \R^d$ the
  associated self-conformal set containing at least two points. If $F$ does not satisfy the identity
  limit criterion, then there exists a constant $C \ge 1$ such that for every $\eps>0$ there are
  $0<\delta<\eps$ and $\iii,\jjj \in \Sigma_*$ for which
  \begin{equation*}
    C^{-1}\delta \max\{\|\fii_\iii'\|,\|\fii_\jjj'\|\} \le |\fii_\iii(x)-\fii_\jjj(x)| \le C\delta \min\{\|\fii_\iii'\|,\|\fii_\jjj'\|\}
  \end{equation*}
  for all $x \in V$.
\end{lemma}

\begin{proof}
  By the assumption, for every $\eps>0$ there are $\iii,\jjj \in \Sigma_*$ such that
  \begin{equation} \label{eq:not-ILC}
    0<\sup_{x \in F}|\fii_\iii(x)-\fii_\jjj(x)| \le \eps\max\{\|\fii_\iii'\|,\|\fii_\jjj'\|\}.
  \end{equation}
  Let $\alpha,c>0$ be as in Lemma \ref{thm:derivatives-holder}. Recalling that $V \supset F$ is open, we see that there exists $\eps_0>0$ such that $\eps_0^{1/(1+\alpha)} < \diam(F)$ and
  $B(x,\eps_0^{1/(1+\alpha)}) \subset V$ for all $x \in F$. Fix $0<\eps<\eps_0$ and let $\iii,\jjj
  \in \Sigma_*$ be such that \eqref{eq:not-ILC} holds. By compactness of $F$, the supremum in
  \eqref{eq:not-ILC} is attained by some $x_0 \in F$. To simplify notation, write
  $f(x)=\fii_\iii(x)-\fii_\jjj(x)$ and $\delta = |f(x_0)|/\max\{\|\fii_\iii'\|,\|\fii_\jjj'\|\}$.
  Note that 
  \begin{equation} \label{eq:delta}
    |f(x)| \le |f(x_0)| = \delta\max\{\|\fii_\iii'\|,\|\fii_\jjj'\|\} \le \eps\max\{\|\fii_\iii'\|,\|\fii_\jjj'\|\}
  \end{equation}
  for all $x \in F$ and, in particular, $0<\delta \le \eps$. 

  By the triangle inequality and Lemma \ref{thm:derivatives-holder}, we obtain
  \begin{equation} \label{eq:close-maps1}
  \begin{split}
    ||f'(y)|-|f'(x_0)|| &\le |f'(y) - f'(x_0)| \le |\fii_\iii'(y) - \fii_\iii'(x_0)| + |\fii_\jjj'(y) - \fii_\jjj'(x_0)| \\
    &\le c(\|\fii_\iii'\| + \|\fii_\jjj'\|)|y-x_0|^\alpha \le 2c\max\{\|\fii_\iii'\|,\|\fii_\jjj'\|\}|y-x_0|^\alpha
  \end{split}
  \end{equation}
  for all $y \in V$. Let $H \ge 1$ be as in \eqref{eq:unif-perfect}. We will next show that
  \begin{equation} \label{eq:close-maps2}
    |f'(x_0)| \le (3H+2c) \delta^{\alpha/(1+\alpha)} \max\{\|\fii_\iii'\|,\|\fii_\jjj'\|\}.
  \end{equation}
  To prove \eqref{eq:close-maps2}, we assume the opposite inequality for a contradiction. Since $F$ contains at least two points, it follows from Lemma \ref{thm:unif-perfect} that $F$ is uniformly perfect and there exists a point $z \in F \cap B(x_0,\delta^{1/(1+\alpha)}) \setminus B(x_0,\delta^{1/(1+\alpha)}/H)$. By convexity of $V$, the line connecting $x_0$ and $z$ is contained in $V$ and hence, $z_t = (1-t)x_0 +t z \in V$ for all $t \in [0,1]$. Recalling \eqref{eq:close-maps1}, we have
  \begin{equation} \label{eq:explicitder}
    |f'(z_t)-f'(x_0)| \leq 2c\max\{\|\fii_{\iii}'\|,\|\fii_{\jjj}'\|\}|z_t-x_0|^\alpha
    \leq 2c \delta^{\alpha/(1+\alpha)}\max\{\|\fii_{\iii}'\|,\|\fii_{\jjj}'\|\}.
  \end{equation}
  Define unit vectors $u$ and $v$ by setting $u = (z-x_0)/|z-x_0|$ and $v = f'(x_0){u}/|f'(x_0)|$. As $f'(y)$ is a similarity for all $y \in V$, we have $\la f'(x_0)u,v \ra = |f'(x_0)u|^2/|f'(x_0)| = |f'(x_0)|$ and $|f'(x_0)u-f'(z_t)u| = |f'(x_0)-f'(z_t)|$. Therefore, by the Cauchy-Schwarz inequality, the assumption that \eqref{eq:close-maps2} does not hold, and \eqref{eq:explicitder}, we have
  \begin{equation} \label{eq:productEst}
  \begin{split}
    \la f'(z_t)u,v \ra &= \la f'(x_0)u,v \ra - \la f'(x_0)u - f'(z_t)u,v \ra \\
    &\ge |f'(x_0)| - |f'(x_0)-f'(z_t)|
    \ge 3H \delta^{\alpha/(1+\alpha)} \max\{\|\fii_\iii'\|,\|\fii_\jjj'\|\}  
  \end{split}
  \end{equation}
  for all $t \in [0,1]$. Since, by conformality, $\la \nabla \la f(y),v \ra, u \ra = \la f'(y)u,v \ra$ for all $y \in V$, the fundamental theorem of calculus and the multivariate chain rule imply
  \begin{equation} \label{eq:gradient-rule}
  \begin{split}
    \la f(z),v \ra - \la f(x_0),v \ra &= \int_0^1 \tfrac{\mathrm{d}}{\mathrm{d}t} \la f(z_t),v \ra \,\mathrm{d}t
    = \int_0^1 \la \nabla \la f(z_t),v \ra, \tfrac{\mathrm{d}}{\mathrm{d}t} z_t \ra \,\mathrm{d}t \\ 
    &= |z-x_0| \int_0^1 \la \nabla \la f(z_t),v \ra, u \ra \,\mathrm{d}t
    = |z-x_0| \int_0^1 \la f'(z_t)u,v \ra \,\mathrm{d}t.
  \end{split}
  \end{equation}
  Hence, by the Cauchy-Schwarz inequality, \eqref{eq:gradient-rule}, and \eqref{eq:productEst},
  \begin{align*}
    |f(z)| &\ge \la f(z),v \ra = \la f(x_0),v \ra + |z-x_0| \int_0^1 \la f'(z_t)u,v \ra \,\mathrm{d}t \\ 
    &\ge -|f(x_0)| + \frac{\delta^{1/(1+\alpha)}}{H} 3H \delta^{\alpha/(1+\alpha)}
    \max\{\|\fii_\iii'\|,\|\fii_\jjj'\|\} \\
    &= 2\delta \max\{\|\fii_\iii'\|,\|\fii_\jjj'\|\} > |f(x_0)|.
  \end{align*}
  As this contradicts \eqref{eq:delta}, i.e.\ the maximality of $x_0$, we have shown \eqref{eq:close-maps2}.

  Combining \eqref{eq:close-maps1} and \eqref{eq:close-maps2}, we see that
  \begin{equation} \label{eq:close-maps3}
  \begin{split}
    |f'(y)| &\le |f'(x_0)| + 2c\max\{\|\fii_\iii'\|,\|\fii_\jjj'\|\}|y-x_0|^\alpha \\
    &\le (3H+4c)\delta^{\alpha/(1+\alpha)}\max\{\|\fii_\iii'\|,\|\fii_\jjj'\|\}
  \end{split}
  \end{equation}
  for all $y \in B(x_0,\delta^{1/(1+\alpha)})$. Write $r = \delta^{1/(1+\alpha)}/(6H+8c) \le \delta^{1/(1+\alpha)}$, fix $x \in B(x_0,r)$, and define $y_t = (1-t)x_0+tx$ for all $t \in [0,1]$. By the fundamental theorem of calculus, there exists $y \in B(x_0,r)$ such that, by \eqref{eq:close-maps3},
  \begin{equation} \label{eq:close-maps4}
  \begin{split}
    ||f(x)|-|f(x_0)|| &\le |f(x)-f(x_0)| = \biggl| \int_0^1 f'(y_t) \tfrac{\mathrm{d}}{\mathrm{d}t}y_t \,\mathrm{d}t \biggr| \leq |f'(y)||x-x_0| \\
    &\le (3H+4c)\delta^{\alpha/(1+\alpha)}\max\{\|\fii_\iii'\|,\|\fii_\jjj'\|\} r = \tfrac12 \delta \max\{\|\fii_\iii'\|,\|\fii_\jjj'\|\}.
  \end{split}
  \end{equation}
  Now \eqref{eq:delta} and \eqref{eq:close-maps4} imply
  \begin{equation} \label{eq:close-maps5}
    \tfrac12 \delta \max\{\|\fii_\iii'\|,\|\fii_\jjj'\|\} = |f(x_0)| - \tfrac12 \delta
    \max\{\|\fii_\iii'\|,\|\fii_\jjj'\|\} \le |f(x)| \le \delta \max\{\|\fii_\iii'\|,\|\fii_\jjj'\|\}
  \end{equation}
  for all $x \in B(x_0,r)$.

  Let $\kkk \in \Sigma$ be such that $\lim_{n\to\infty} \fii_{\kkk|_n}(x)=x_0$ for any
  $x\in V$ and choose $n \in \N$ such that $\diam(\fii_{\kkk|_n}(V))<r$ and
  $\diam(\fii_{\kkk|_{n-1}}(V))\ge r$. Note that $\fii_{\kkk|_n}(V) \subset B(x_0,r)$ and hence
  \eqref{eq:close-maps5} holds for all points in $\fii_{\kkk|_n}(V)$. By \eqref{eq:semiconformal-multi}, we
  have $K^{-2}\|\fii_\hhh'\|\|\fii_{\kkk|_n}'\| \le \|\fii_{\hhh\kkk|_n}'\| \le
  \|\fii_\hhh'\|\|\fii_{\kkk|_n}'\|$ for all $\hhh \in \Sigma_*$. Observe that, by
  \eqref{eq:semiconformal-diam},
  \begin{equation*}
    \frac{K^{-2}\min_{i \in \{1,\ldots,N\}}\|\fii_i'\|}{\diam(F)}r \le \|\fii_{\kkk|_n}'\| \le \frac{K}{\diam(F)}r.
  \end{equation*}
  Therefore, by \eqref{eq:close-maps5},
  \begin{align*}
    |f(\fii_{\kkk|_n}(x))| &\le \delta\max\{\|\fii_\iii'\|,\|\fii_\jjj'\|\} \le \delta
    K^2\|\fii_{\kkk|_n}'\|^{-1} \max\{\|\fii_{\iii\kkk|_n}'\|,\|\fii_{\jjj\kkk|_n}'\|\} \\
    &\le \frac{\delta K^4\diam(F)}{r\min_{i \in \{1,\ldots,N\}}\|\fii_i'\|}
    \max\{\|\fii_{\iii\kkk|_n}'\|,\|\fii_{\jjj\kkk|_n}'\|\} \\
    &\le \delta^{\alpha/(1+\alpha)} \frac{K^4(4H+2^{2+\alpha})\diam(F)}{\min_{i \in
      \{1,\ldots,N\}}\|\fii_i'\|} \max\{\|\fii_{\iii\kkk|_n}'\|,\|\fii_{\jjj\kkk|_n}'\|\}
  \end{align*}
  and
  \begin{equation*}
    |f(\fii_{\kkk|_n}(x))| \ge \delta^{\alpha/(1+\alpha)} K^{-1}(2H+2^{1+\alpha})\diam(F)
    \max\{\|\fii_{\iii\kkk|_n}'\|,\|\fii_{\jjj\kkk|_n}'\|\}
  \end{equation*}
  for all $x \in V$. Writing
  \begin{equation*}
    C = \max\biggl\{ \frac{K^4(4H+2^{2+\alpha})\diam(F)}{\min_{i \in \{1,\ldots,N\}}\|\fii_i'\|}, \frac{K}{(2H+2^{1+\alpha})\diam(F)} \biggr\},
  \end{equation*}
  we have thus shown that
  \begin{equation} \label{eq:almost-finished}
  \begin{split}
    C^{-1}\delta^{\alpha/(1+\alpha)} \max\{\|\fii_{\iii\kkk|_n}'\|,\|\fii_{\jjj\kkk|_n}'\|\} &\le
    |\fii_{\iii\kkk|_n}(x)-\fii_{\jjj\kkk|_n}(x)| \\
    &\le C\delta^{\alpha/(1+\alpha)}
    \max\{\|\fii_{\iii\kkk|_n}'\|,\|\fii_{\jjj\kkk|_n}'\|\}
  \end{split}
  \end{equation}
  for all $x \in V$.

  To finish the proof, fix $0<\eps'<\diam(F)/(4KC)$ and take $0<\eps<\eps_0$ such that
  $\eps^{\alpha/(1+\alpha)}<\eps'$. Let $0<\delta\le\eps$, $\iii' = \iii\kkk|_n$, and $\jjj' =
  \jjj\kkk|_n$ be so that \eqref{eq:almost-finished} holds and define $\delta' =
  \delta^{\alpha/(1+\alpha)} \le \eps^{\alpha/(1+\alpha)} < \eps'$. Notice that, by
  \eqref{eq:semiconformal-diam}, the triangle inequality, and \eqref{eq:almost-finished},
  \begin{align*}
    \|\fii_{\jjj'}'\| &\le \frac{K}{\diam(F)}\diam(\fii_{\jjj'}(F)) \\
    &\le \frac{K}{\diam(F)}\bigl(\diam(\fii_{\iii'}(F)) + 2C\delta'\max\{\|\fii_{\iii'}'\|,\|\fii_{\jjj'}'\|\}\bigr) \\
    &\le K\|\fii_{\iii'}'\| + \frac{2KC\eps'}{\diam(F)}\max\{\|\fii_{\iii'}'\|,\|\fii_{\jjj'}'\|\}.
  \end{align*}
  Therefore, if $\|\fii_{\jjj'}'\| \ge \|\fii_{\iii'}'\|$, we have
  \begin{equation*}
    \|\fii_{\jjj'}'\| \le \frac{K\diam(F)}{\diam(F)-2KC\eps'}\|\fii_{\iii'}'\| \le 2K\|\fii_{\iii'}'\|
  \end{equation*}
  and similarly the other way around. By \eqref{eq:almost-finished}, we now have
  \begin{equation*}
    C^{-1}\delta'\max\{\|\fii_{\iii'}'\|,\|\fii_{\jjj'}'\|\} \le |\fii_{\iii'}(x)-\fii_{\jjj'}(x)| \le 2KC\delta'\min\{\|\fii_{\iii'}'\|,\|\fii_{\jjj'}'\|\}
  \end{equation*}
  for all $x\in V$, which is what we wanted to show.
\end{proof}

\begin{proposition} \label{thm:WSC-implies-ILC}
  Let $F \subset \R^d$ be a self-conformal set containing at least two points. If $F$ satisfies the
  weak separation condition, then it satisfies the identity limit criterion.
\end{proposition}

\begin{proof}
  Suppose to the contrary that $F$ does not satisfy the identity limit criterion. Let $C \ge 1$ be
  as in Lemma \ref{thm:close-maps} and $K \ge 1$ as in Lemma \ref{thm:semiconformal}. 
  For each $q \in \N$ write $\lll(q) = 1 \cdots 1 \in \Sigma_q$ and $\eps(q) = \tfrac{2}{3} C K^{-2}\|\fii'_{\lll(q)}\|\diam(F)/q > 0$. Choose $q\in\N$ to be the smallest integer for which
  \begin{equation} \label{eq:choice-of-q}
    \frac{K}{\diam(F)} \max_{\jjj \in \Sigma_q} \diam(\fii_\jjj(F)) < \frac{3q-2}{3q+2} =
    \frac{C\|\fii'_{\lll(q)}\|\diam(F)-K^2\eps(q)}{C\|\fii'_{\lll(q)}\|\diam(F)+K^2\eps(q)}.
  \end{equation}
  We will prove that $F$ does not satisfy the weak separation condition by showing that for each $n \in \N$ there exist $x \in F$ and $r>0$ such that $\# \Phi(x,r) \ge \lceil n/q \rceil$.

  Fix $n \in \N$ and write $\eps_1 = \eps(q)$. Since $F$ contains at least two points and does not satisfy the identity limit criterion, Lemma \ref{thm:close-maps} implies that there exist $0<\delta_1<\eps_1$ and $\iii_1,\jjj_1 \in \Sigma_*$ such that
  \begin{equation*}
    C^{-1}\delta_1 \|\fii_{\iii_1}'\| \le |\fii_{\iii_1}(x)-\fii_{\jjj_1}(x)| \le C\delta_1 \|\fii_{\iii_1}'\|
  \end{equation*}
  for all $x \in V$. We will choose $\delta_k>0$ and $\iii_k,\jjj_k \in \Sigma_*$, $k \in
  \{1,\ldots,n\}$, inductively. Assuming $0<\delta_{k-1}<\eps_{k-1}<1$ and $\iii_{k-1},\jjj_{k-1}
  \in \Sigma_*$ have already been chosen for some $k \in \{2,\ldots,n\}$, let us fix $0 < \eps_{k} <
  (2K^5C^2)^{-1}\delta_{k-1}\|\fii_{\iii_{k-1}}'\|$. By Lemma \ref{thm:close-maps}, we then choose
  $0<\delta_{k}<\eps_{k}$ and $\iii_{k},\jjj_{k} \in \Sigma_*$ such that
  \begin{equation} \label{eq:WSC-ILC1}
    C^{-1}\delta_{k} \|\fii_{\iii_{k}}'\| \le |\fii_{\iii_{k}}(x)-\fii_{\jjj_{k}}(x)| \le C\delta_{k} \|\fii_{\iii_{k}}'\|
  \end{equation}
  for all $x \in V$.

  Define $\iii = \iii_n\cdots\iii_1$ and $\kkk_m =
  \iii_n\cdots\iii_{m+1}\jjj_m\iii_{m-1}\cdots\iii_1$ for all $m \in \{1,\ldots,n\}$. Fix $m,l \in
  \{1,\ldots,n\}$ such that $m \ne l$ and notice that we may assume $l<m$, relabeling if necessary. We claim that
  \begin{equation} \label{eq:not-same-with-lll}
    \fii_{\kkk_m\mmm}|_F \ne \fii_{\kkk_l\mmm}|_F
  \end{equation}
  for all $\mmm \in \Sigma_*$.
  By \eqref{eq:semiconformal-multi}, we have $K^{-2}\|\fii_\kkk'\|\|\fii_\jjj'\| \le \|\fii_{\kkk\jjj}'\|
  \le \|\fii_\kkk'\|\|\fii_\jjj'\|$ for all $\kkk,\jjj \in \Sigma_*$. Therefore, by
  \eqref{eq:semiconformal} and \eqref{eq:WSC-ILC1}, we have
  \begin{align*}
    |\fii_\iii(x)-\fii_{\kkk_l}(x)| &\ge K^{-1} \|\fii_{\iii_n\cdots\iii_{l+1}}'\| |\fii_{\iii_l}(\fii_{\iii_{l-1}\cdots\iii_1}(x)) - \fii_{\jjj_l}(\fii_{\iii_{l-1}\cdots\iii_1}(x))| \\
    &\ge (KC)^{-1} \delta_l \|\fii_{\iii_n\cdots\iii_{l+1}}'\| \|\fii_{\iii_l}'\| \ge (KC)^{-1} \delta_l \|\fii_{\iii_n\cdots\iii_l}'\|
  \end{align*}
  and, as $\delta_m < \eps_m < (2K^5C^2)^{-1}\delta_{m-1}\|\fii_{\iii_{m-1}}'\| < \cdots < (2K^5C^2)^{l-m} \delta_l \|\fii_{\iii_{m-1}}'\| \cdots \|\fii_{\iii_{l}}'\|$, also
  \begin{equation} \label{eq:WSC-ILC2}
  \begin{split}
    |\fii_\iii(x)-\fii_{\kkk_m}(x)| &\le K^2C \delta_m \|\fii_{\iii_n\cdots\iii_{m}}'\| \le K^2C (2K^5C^2)^{l-m} \delta_l \|\fii_{\iii_n\cdots\iii_{m}}'\| \|\fii_{\iii_{m-1}}'\| \cdots \|\fii_{\iii_{l}}'\| \\
    &\le K^2C (2K^3C^2)^{l-m} \delta_l \|\fii_{\iii_n\cdots\iii_{l}}'\| \le (2KC)^{-1} \delta_l \|\fii_{\iii_n\cdots\iii_{l}}'\|
  \end{split}
  \end{equation}
  for all $x \in V$. Since now
  \begin{align*}
    |\fii_{\kkk_l}(x) - \fii_{\kkk_m}(x)| &\ge
    ||\fii_\iii(x)-\fii_{\kkk_l}(x)|-|\fii_\iii(x)-\fii_{\kkk_m}(x)|| \nonumber\\
    &\ge (KC)^{-1} \delta_l \|\fii_{\iii_n\cdots\iii_l}'\| - (2KC)^{-1} \delta_l
    \|\fii_{\iii_n\cdots\iii_{l}}'\| \nonumber\\
    &= (2KC)^{-1} \delta_l \|\fii_{\iii_n\cdots\iii_{l}}'\| > 0
  \end{align*}
  for all $x \in V$, we see that $|\fii_{\kkk_l\mmm}(x) - \fii_{\kkk_m\mmm}(x)| > 0$ for all $x \in \fii_\mmm(V)$ and $\mmm \in \Sigma_*$. Therefore, \eqref{eq:not-same-with-lll} holds and, in particular, the set
  \begin{equation} \label{eq:Phi-has-n-elements}
    \Psi_p := \{ \fii_{\kkk_m\lll(p)}|_F : m \in \{1,\ldots,n\} \}
  \end{equation}
  has $n$ elements for all $p \in \{1,\ldots,q\}$.
  
  Let $r = \max_{m\in\{1,\ldots,n\}}\diam(\fii_{\kkk_m\lll(q)}(F))$ and $x = \fii_{\iii\lll(q)}(x_0)$, where $x_0
  \in F$. We will next show that
  \begin{equation} \label{eq:prop75-goal}
    \diam(\fii_{\kkk_m\lll(q)}(F)) \le r < \diam(\fii_{\kkk_m}(F)) \quad \text{and} \quad
    \fii_{\kkk_m\lll(q)}(F) \cap B(x,r) \ne \emptyset
  \end{equation}
  for all $m \in \{1,\ldots,n\}$. To that end, fix $m \in \{1,\ldots,n\}$. Choosing $y,z \in F$ such that
  $|\fii_{\kkk_m\lll(q)}(y)-\fii_{\kkk_m\lll(q)}(z)| = \diam(\fii_{\kkk_m\lll(q)}(F))$, we see, by \eqref{eq:semiconformal}, \eqref{eq:semiconformal-diam}, and \eqref{eq:choice-of-q}, that
  \begin{equation} \label{eq:WSC-ILC6}
  \begin{split}
    \diam(\fii_{\kkk_m\lll(q)}(F)) &\le \|\fii_{\kkk_m}'\| |\fii_{\lll(q)}(y)-\fii_{\lll(q)}(z)| \\
    &\le \frac{K}{\diam(F)} \diam(\fii_{\kkk_m}(F)) \diam(\fii_{\lll(q)}(F)) \\
    &< \frac{C\|\fii'_{\lll(q)}\|\diam(F)-K^2\eps(q)}{C\|\fii'_{\lll(q)}\|\diam(F)+K^2\eps(q)}
    \diam(\fii_{\kkk_m}(F)).
  \end{split}
  \end{equation}
  Note that \eqref{eq:WSC-ILC2} implies
  \begin{equation*}
    |\fii_\iii(x)-\fii_{\kkk_l}(x)| \le (2KC)^{-1} \eps(q) \|\fii_{\iii}'\|
  \end{equation*}
  for all $x \in V$ and $l \in \{1,\ldots,n\}$. Therefore, by the triangle inequality, \eqref{eq:semiconformal-multi}, and \eqref{eq:semiconformal-diam},
  \begin{equation} \label{eq:WSC-ILC7}
  \begin{split}
    \diam(\fii_{\kkk_l\lll(q)}(F)) &\le \diam(\fii_{\iii\lll(q)}(F)) + (KC)^{-1}\eps(q)\|\fii_\iii'\| \\
    &\le \biggl(1+\frac{K^2\eps(q)}{C\|\fii_{\lll(q)}'\|\diam(F)} \biggr) \diam(\fii_{\iii\lll(q)}(F))
  \end{split}
  \end{equation}
  for all $l \in \{1,\ldots,n\}$. Since, similarly,
  \begin{equation} \label{eq:WSC-ILC8}
    \diam(\fii_{\iii\lll(q)}(F)) \le \diam(\fii_{\kkk_m\lll(q)}(F)) +
    \frac{K^2\eps(q)}{C\|\fii_{\lll(q)}'\|\diam(F)}
    \diam(\fii_{\iii\lll(q)}(F)),
  \end{equation}
  we conclude, by \eqref{eq:WSC-ILC6}, \eqref{eq:WSC-ILC8}, and \eqref{eq:WSC-ILC7}, that
  \begin{equation*}
    \diam(\fii_{\kkk_m}(F)) > \max_{l\in\{1,\ldots,n\}}\diam(\fii_{\kkk_l\lll(q)}(F)) = r \ge
    \diam(\fii_{\kkk_m\lll(q)}(F))
  \end{equation*}
  as desired. Observe that the role of $q$ is to guarantee the strict inequality above -- because of
  conformality, it might happen that $\diam(\fii_\kkk(F)) = \diam(\fii_{\kkk^-}(F))$ for some $\kkk
  \in \Sigma_*$; see Example~\ref{ex:shortword}. Finally, note that \eqref{eq:WSC-ILC2}, the choice of $\eps(q)$, \eqref{eq:semiconformal-multi}, \eqref{eq:semiconformal-diam}, and \eqref{eq:WSC-ILC8} give us
  \begin{align*}
    |x-\fii_{\kkk_m\lll(q)}(x_0)| &= |\fii_{\iii}(\fii_{\lll(q)}(x_0))-\fii_{\kkk_m}(\fii_{\lll(q)}(x_0))| \le (2KC)^{-1}\eps(q)\|\fii_\iii'\| \\
    &= \tfrac13 K^{-3}\|\fii'_{\lll(q)}\|\|\fii_\iii'\|\diam(F) \le \tfrac13 K^{-1}\|\fii'_{\iii\lll(q)}\|\diam(F)\\
    &\le \tfrac13 \diam(\fii_{\iii\lll(q)}(F))\leq r
  \end{align*}
  yielding $\fii_{\kkk_m\lll(q)}(F) \cap B(x,r) \ne \emptyset$ as desired.

  Because of the length difference $q$, we cannot directly apply \eqref{eq:prop75-goal} in the definition of the weak separation condition. But relying on \eqref{eq:prop75-goal}, we see that for each $m \in \{1,\ldots,n\}$ there is $p_m \in \{1,\ldots, q\}$ such that $\diam(\fii_{\kkk_m\lll(p_m)}(F)) \le r < \diam(\fii_{\kkk_m\lll(p_m-1)}(F))$ and $\fii_{\kkk_m\lll(p_m)}(F) \cap B(x,r) \ne \emptyset$. Hence,
  \begin{equation*}
    \Phi_p := \{ \fii_{k_m\lll(p)}|_F : m \in \{1,\ldots,n\} \text{ and } p_m = p \} \subset \Phi(x,r)
  \end{equation*}
  for all $p \in \{1,\ldots,q\}$. By \eqref{eq:Phi-has-n-elements}, we have $\Phi_p \subset \Psi_p$, $\#\Phi_p = \#\{m \in \{1,\ldots, n\} : p_m=p\}$, and $\#\Psi_p = n$ for all $p \in \{1,\ldots,q\}$. Since the function $m \mapsto p_m$ is from $\{1,\dots,n\}$ to $\{1,\dots,q\}$, there exists $p \in \{1,\ldots, q\}$ such that $\#\Phi_p \geq \lceil n/q \rceil$. Therefore, we have shown that $\#\Phi(x,r) \ge \lceil n/q \rceil$ and finished the proof.
\end{proof}

We finish the section with two examples. The first one verifies the need to use $q$ in \eqref{eq:prop75-goal} and the second one examines the difference between the original definition of the weak separation condition and our definition.

\begin{example}\label{ex:shortword}
  We exhibit a conformal iterated function system $\{\fii_1,\fii_2,\fii_3\}$ on $\R^2$ for which the associated self-conformal set $F \subset \R^2$ satisfies $\diam(\fii_\iii(F)) = \diam(\fii_{\iii^-}(F))$ for $\iii = 32 \in \Sigma_*$.

  Using complex notation, we define
  \begin{equation*}
    \fii_1(z) = \tfrac{1}{1000}z - \tfrac{9}{10}, \quad
    \fii_2(z) = \tfrac{19}{20}iz, \quad
    \fii_3(z) = \tfrac{z}{2(z-2i)}
  \end{equation*}
  for all $z \in \C$. The mapping $\fii_1$ is a strongly contracting homothety, $\fii_2$ is a weakly contracting similarity that involves a rotation by $\tfrac{\pi}{2}$, and $\fii_3$ is a M\"obius transformation with singularity at $2i$. Therefore, all the mappings are injective and holomorphic on $\C \setminus \{2i\}$. To see that their collection is a conformal iterated function system, it is enough to verify that there exists a bounded open convex set $\Omega \subset \C$ such that $\overline{\fii_j(\Omega)} \subset \Omega$ and $\|\fii_j'\| = \sup_{z \in \Omega}|\fii_j'(z)| < 1$ for all $j \in \{1,2,3\}$.

  Write $r_0 = \tfrac{901}{1000}$ and define $\Omega = B^o(0,r_0)$, where $B^o(z,r)$ is an open ball
  centered at $z \in \C$ with radius $r>0$. Note that the singularity $2i$ is not contained in the
  closure of $\Omega$ and hence, each of the mappings $\fii_j$ maps balls in $\Omega$ onto balls. A
  simple calculation shows that $\fii_1(\Omega) = B^o(c_1,r_1)$, where $c_1 = -\tfrac{9}{10}$ and
  $r_1 = \tfrac{901}{10^6}$. Since $|c_1-r_1|<r_0$, we see that $\overline{\fii_1(\Omega)} \subset
  \Omega$. Similarly, we see that $\overline{\fii_2(\Omega)} = B(0,\tfrac{19}{20}r_0) \subset
  \Omega$. We determine $\fii_3(\Omega)$ by looking at the images of $r_0$, $ir_0$, and $-r_0$ from
  the boundary of $\Omega$. Indeed, these three points uniquely describe a circle and hence the ball
  $\fii_3(\Omega) = B^o(c_3,r_3)$. The center point $c_3 = -\tfrac{811801}{6376398}$ can be
  calculated from the equations
  \begin{equation*}
    |c_3-\fii_3(r_0)| = |c_3-\fii_3(ir_0)| = |c_3-\fii_3(-r_0)|,
  \end{equation*}
  where each of the value is the radius $r_3 = \tfrac{901000}{3188199}$. Since $|c_3|+r_3 < r_0$, we
  see that also $\overline{\fii_3(\Omega)} \subset \Omega$. Furthermore, a direct calculation shows
  that $\fii_1'(z) = \tfrac{1}{1000}$, $\fii_2'(z) = \tfrac{19}{20}i$, and $\fii_3'(z) =
  -\tfrac{i}{(z-2i)^2}$ for all $z \in \C \setminus \{2i\}$. Therefore, as $|\fii_3(z)| \le
  (2-r_0)^{-2} < 1$ for all $z \in \Omega$, we have $\|\fii_j'\| < 1$ for all $j \in \{1,2,3\}$. The
  collection $\{\fii_1,\fii_2,\fii_3\}$ is thus a conformal iterated function system. Let $F \subset
  \C$ be the associated self-conformal set.

  Since $\fii_{32}(F) \subset \fii_3(F)$, to see that the diameters are equal, it suffices to prove that $\diam(\fii_3(F)) \le \diam(\fii_{32}(F))$. Let $w = -\tfrac{100}{111} \in F$ be the fixed point of $\fii_1$. Defining $q_1 = \fii_{32}(w) = \tfrac{95}{634} \in F$ and $q_2 = \fii_{3222}(w) = -\tfrac{6859}{21802} \in F$, it follows that
  \begin{equation*}
    \diam(\fii_{32}(F)) \ge |q_1-q_2| = \tfrac{1604949}{3455617}.
  \end{equation*}
  Showing that this number is an upper bound for $\diam(\fii_3(F))$ will thus finish the proof. Calculating as before, we see that $\fii_1(B(0,w)) = B(-\tfrac{9}{10},\tfrac{1}{1110}) \subset B(0,w)$, $\fii_2(B(0,w)) = B(0,\tfrac{95}{111}) \subset B(0,w)$, and $\fii_3(B(0,w)) = B(-\tfrac{1250}{9821},\tfrac{2775}{9821}) \subset B(0,w)$. Therefore, $F \subset B(0,w)$. Write $\Gamma = \{31, 33, 321, 323, 3221, 3223, 32221, 32222, 32223\} \subset \Sigma_*$ and note that $\fii_3(F) \subset \bigcup_{\iii \in \Gamma} \fii_\iii(B(0,w))$. For each $\iii \in \Gamma$, let $c_\iii$ be the center and $r_\iii$ the radius of the ball $\fii_\iii(B(0,w))$. Numerical calculations show that
  \begin{equation*}
    \diam(\fii_3(F)) \le \diam\biggl( \bigcup_{\iii\in\Gamma} \fii_\iii(B(0,w)) \biggr) = \max_{\iii,\jjj\in\Gamma} \{|c_\iii-c_\jjj|+r_\iii+r_\jjj\} = \tfrac{1604949}{3455617}
  \end{equation*}
  as required.
\end{example}

\begin{example} \label{ex:lau-ngai-wang-ex}
  We exhibit a conformal iterated function system $\{\fii_1,\fii_2,\fii_3\}$ on $\R$ for which the associated self-conformal set $F \subset \R$ satisfies the weak separation condition but has
  \begin{equation} \label{eq:lau-ngai-wang-wsc}
    \sup\{\#\Phi^*(x,r) : x \in F \text{ and } r>0\} = \infty,
  \end{equation}
  where
  \begin{equation*}
    \Phi^*(x,r) = \{ \fii_\iii : \diam(\fii_\iii(F)) \le r < \diam(\fii_{\iii^-}(F)) \text{ and } \fii_\iii(F) \cap B(x,r) \ne \emptyset \}
  \end{equation*}
  for all $x \in \R$ and $r>0$. Since, by \cite[Remark 3.7(1)]{KaenmakiRossi2016}, the condition
  $\sup\{\#\Phi^*(x,r) : x \in F \text{ and } r>0\} < \infty$ is equivalent to the original
  definition of Lau, Ngai, and Wang \cite{LauNgaiWang2009}, we see that our definition is strictly
  weaker in the non-analytic case.

  Let
  \begin{equation*}
    g(x) =
    \begin{cases}
      \tfrac{1}{180}(9x^2 - 6x + 1),    &\text{if } \tfrac{1}{3} < x < \tfrac{5}{12}, \\
      -\tfrac{1}{180}(9x^2 - 9x + \tfrac{17}{8}), &\text{if } \tfrac{5}{12} \le x < \tfrac{7}{12}, \\
      \tfrac{1}{120}(6x^2 -8x + \tfrac{8}{3}),   &\text{if } \tfrac{7}{12} \le x < \tfrac{2}{3}, \\
      0, &\text{otherwise},
    \end{cases}
  \end{equation*}
  and notice that $0<g(x)\le 1/2880$ for all $x \in (\tfrac13, \tfrac23)$ and $g$ is continuously
  differentiable such that $0<g'(x)\le 1/120$ for all $x \in (\tfrac13,\tfrac12)$ and $-1/120\le
  g'(x)<0$ for all $x \in (\tfrac12,\tfrac23)$. In fact, $g'$ is a piecewise linear continuous
  function and hence H\"older continuous. Define
  \begin{equation*}
    \fii_1(x) = \tfrac13 x, \quad
    \fii_2(x) = \tfrac13 x + \tfrac23, \quad
    \fii_3(x) = \tfrac13 x + g(x),
  \end{equation*}
  for all $x \in \R$ and notice that each $\fii_i$ is a strictly increasing $C^{1+\alpha}$-function.
  The collection $\{\fii_1,\fii_2,\fii_3\}$ is therefore a conformal iterated function system and
  since $\fii_1|_F = \fii_3|_F$, the associated self-conformal set is the standard $\tfrac13$-Cantor set $F$. Note that
  also $\{\fii_1,\fii_2\}$ defines $F$ and it is well known that $F$, defined by these two maps,
  satisfies the open set condition. Therefore, as $\fii_1|_F = \fii_3|_F$,
  the set $F$, defined by all three maps, satisfies the weak separation condition.

  To see that \eqref{eq:lau-ngai-wang-wsc} holds, observe first that $\diam(\fii_\iii(F)) = 3^{-n}$
  for all $\iii \in \Sigma_n$ and $n \in \N$. Let $\iii(k) = i_1i_2\cdots$ be the word in $\Sigma$
  such that $i_k = 3$ and $i_j = 1$ for all $j \in \N \setminus \{k\}$. Note that $0 \in
  \fii_{\iii(k)|_n}(F)$ and
  \begin{equation*}
    \fii_{\iii(k)|_n}(x) = 3^{-n}x + 3^{-k+1}g(3^{-n+k}x)
  \end{equation*}
  for all $x \in \R$, $k \in \{1,\ldots,n\}$, and $n \in \N$. Therefore, $\fii_{\iii(k)|_n} \ne
  \fii_{\iii(m)|_n}$ for all $k,m \in \{1,\ldots,n\}$ with $k \ne m$ and $\Phi^*(0,3^{-n})$ has at
  least $n$ elements for all $n \in \N$.
\end{example}

\section{Proof of Theorem \ref{thm:not-wsc-assouad-one}} \label{sec:not-wsc-assouad-one}

Let $E \subset \R$ be a compact set. For each $x\in \R$ and $r>0$ we define the \emph{magnification} $M_{x,r} \colon \R \to \R$ by setting
\begin{equation*}
  M_{x,r}(z) = \frac{z-x}{r}
\end{equation*}
for all $z \in \R$. We say that $T \subset [-1,1]$ is a \emph{weak tangent} of $E$ if there exist sequences $(x_n)_{n\in\N}$ of points in $\R$ and $(r_n)_{n \in \N}$ of positive real numbers such that $M_{x_n,r_n}(E) \cap [-1,1] \to T$ in Hausdorff distance. Recall that a sequence $(E_n)_{n\in\N}$ of closed subsets of $[-1,1]$ converges to $T$ in Hausdorff distance if
\begin{equation*}
  \lim_{n \to \infty} \sup_{x \in E_n} \dist(x,T) = 0
\quad\text{and}\quad
  \lim_{n \to \infty} \sup_{y \in T} \dist(y,E_n) = 0.
\end{equation*}
If $T$ is a weak tangent of $E$, then it is straightforward to see that $\dimh(T) \le \dima(E)$; see \cite[Proposition 6.1.5]{MackayTyson}.

\begin{proof}[Proof of Theorem \ref{thm:not-wsc-assouad-one}]
  By the above discussion, it suffices to show that there is a constant $D' \ge 1$ such that for every $n \in \N$ there exist $x \in \R$, $r>0$, and points $x_n < x_{n-1} < \cdots < x_1$ in $F$ such that $M_{x,r}(x_n)=-1$, $M_{x,r}(x_1)=1$, and
  \begin{equation*}
    M_{x,r}(x_{k}) - M_{x,r}(x_{k+1}) \le \frac{D'}{n+1}
  \end{equation*}
  for all $k \in \{1,\ldots,n-1\}$. Indeed, by letting $n \to \infty$, this implies that $[-1,1]$ is a weak tangent of $F$ and therefore, $F$ has full Assouad dimension.

  Let $C \ge 1$ be as in Lemma \ref{thm:close-maps} and $K \ge 1$ as in Lemma \ref{thm:semiconformal}. Define
  \begin{equation} \label{eq:def-of-D}
    D = \frac{K^{11}C}{\diam(F)\min_{\iii\in\Sigma_2}\|\fii_\iii'\|},
  \end{equation}
  fix $n \in \N$, and choose $0<\eps<\diam(F)/(4KC)$ such that $(1+2D\eps)^{n-1} \le 2$ and $(1-2D\eps)^{n-1} \ge \tfrac12$. Since $F$ contains at least two points and does not satisfy the identity limit criterion, Lemma \ref{thm:close-maps} implies that there exist $0<\delta_1<\eps$ and $\iii_1',\jjj_1' \in \Sigma_*$ such that
  \begin{equation} \label{eq:not-wsc-assouad-one1}
    C^{-1}\delta_1\max\{\|\fii_{\iii_1'}'\|,\|\fii_{\jjj_1'}'\|\} \le |\fii_{\iii_1'}(x)-\fii_{\jjj_1'}(x)| \le C\delta_1\min\{\|\fii_{\iii_1'}'\|,\|\fii_{\jjj_1'}'\|\}
  \end{equation}
  for all $x \in V$. Recall that, by \eqref{eq:semiconformal}, $K^{-1}\|\fii_{\iii}'\| \le
  |\fii_{\iii}'(x)| \le \|\fii_{\iii}'\|$ for all $x \in V$ and $\iii \in \Sigma_*$. In particular,
  this means that $\fii_\iii'$ is either positive or negative and hence, as $V \subset \R$ is an
  open interval, each $\fii_\iii$ is strictly monotone on $V$. Let $y,z \in F \subset V$ be such
  that $y-z = \diam(F)$. The mean value theorem implies that there exists $w \in V$ such that
  \begin{equation*}
    \fii_{\iii_1'}(y) - \fii_{\iii_1'}(z) = \fii_{\iii_1'}'(w)\diam(F).
  \end{equation*}
  If $\fii_{\iii_1'}'(w) > 0$, then \eqref{eq:not-wsc-assouad-one1} implies
  \begin{align*}
    \fii_{\jjj_1'}(y) - \fii_{\jjj_1'}(z) &\ge \fii_{\iii_1'}(y) - \fii_{\iii_1'}(z) - 2C\delta_1\|\fii_{\iii_1'}'\| \\
    &\ge (\fii_{\iii_1'}'(w) - \tfrac12 K^{-1}\|\fii_{\iii_1'}'\|)\diam(F) \\
    &\ge \tfrac12 K^{-1}\|\fii_{\iii_1'}'\|\diam(F) > 0
  \end{align*}
  and hence, $\fii_{\jjj_1'}(y) > \fii_{\jjj_1'}(z)$ yielding $\fii_{\jjj_1'}'(x) > 0$ for all $x
  \in V$. Similarly, if $\fii_{\iii_1'}'(w) < 0$, then we see that $\fii_{\jjj_1'}'(x) < 0$ for all
  $x \in V$. Therefore, the derivatives $\fii_{\iii_1'}'$ and $\fii_{\jjj_1'}'$ have the same sign.
  Let $\iii_1 = \iii_1'\iii_1'$ and $\jjj_1 = \iii_1'\jjj_1'$ and notice that, by the chain rule,
  $\fii_{\iii_1}'$ and $\fii_{\jjj_1}'$ are positive. By
  \eqref{eq:semiconformal-multi}, \eqref{eq:not-wsc-assouad-one1}, and \eqref{eq:semiconformal}, we have
  \begin{align*}
    (KC)^{-1}\delta_1\max\{\|\fii_{\iii_1}'\|, \|\fii_{\jjj_1}'\|\} &\le K^{-1}\|\fii_{\iii_1'}'\||\fii_{\iii_1'}(x) - \fii_{\jjj_1'}(x)| \\
    &\le |\fii_{\iii_1}(x) - \fii_{\jjj_1}(x)| \\
    &\le \|\fii_{\iii_1'}'\||\fii_{\iii_1'}(x) - \fii_{\jjj_1'}(x)|
    \le K^2C\delta_1\min\{\|\fii_{\iii_1}'\|, \|\fii_{\jjj_1}'\|\}.
  \end{align*}
  Furthermore, since $V \subset \R$ is an open interval and $|\fii_{\iii_1}(x)-\fii_{\jjj_1}(x)| >
  0$ for all $x \in V$, we have, by the intermediate value theorem, that $\fii_{\iii_1}(x) >
  \fii_{\jjj_1}(x)$ for all $x \in V$, relabeling $\iii_1$ and $\jjj_1$ if necessary. Therefore,
  \begin{equation} \label{eq:not-wsc-assouad-one2}
    (KC)^{-1}\delta_1\|\fii_{\iii_1}'\| \le \fii_{\iii_1}(\fii_\kkk(x))-\fii_{\jjj_1}(\fii_\kkk(x)) \le K^2C\delta_1\|\fii_{\iii_1}'\|
  \end{equation}
  for all $x \in V$ and $\kkk \in \Sigma_*$. Notice that, by the chain rule, there exists $\kkk \in \Sigma_2$ such that $\fii_{\kkk}'$ is positive. Choose $\kkk_1 = \kkk \cdots \kkk \in \Sigma_*$ such that
  \begin{equation*}
    \eps\|\fii_{\iii_1}'\|\|\fii_{\kkk_1}'\| < \delta_1\|\fii_{\iii_1}'\| \le \eps\|\fii_{\iii_1}'\|\|\fii_{\kkk_1|_{|\kkk_1|-2}}'\|
  \end{equation*}
  and notice that also $\fii_{\kkk_1}'$ is positive. Therefore, it follows from \eqref{eq:not-wsc-assouad-one2} that
  \begin{equation*}
    (KC)^{-1}\eps\|\fii_{\iii_1\kkk_1}'\| \le \fii_{\iii_1\kkk_1}(x)-\fii_{\jjj_1\kkk_1}(x) \le K^6C\Bigl(\min_{\iii\in\Sigma_2}\|\fii_\iii'\|\Bigr)^{-1}\eps\|\fii_{\iii_1\kkk_1}'\|
  \end{equation*}
  for all $x \in V$.

  To find more points being predefined distance apart, we continue inductively.
  Assuming $\iii_l,\jjj_l,\kkk_l \in \Sigma_*$, $l \in \{1,\ldots,k-1\}$, have already been chosen
  for some $k \in \{2,\ldots,n\}$, we apply Lemma \ref{thm:close-maps} as above to find
  $0<\delta_{k}<\eps K^{-2}\|\fii_{\jjj_{k-1}\kkk_{k-1}\cdots\jjj_1\kkk_1}'\|$ and
  $\iii_{k},\jjj_{k} \in \Sigma_*$ such that $\fii_{\iii_{k}}'$ and $\fii_{\jjj_{k}}'$ are positive, and
  \begin{equation*}
    (KC)^{-1}\delta_{k}\|\fii_{\iii_{k}}'\| \le \fii_{\iii_{k}}(\fii_{\kkk\jjj_{k-1}\kkk_{k-1}\cdots\jjj_1\kkk_1}(x))-\fii_{\jjj_{k}}(\fii_{\kkk\jjj_{k-1}\kkk_{k-1}\cdots\jjj_1\kkk_1}(x)) \le K^2C\delta_{k}\|\fii_{\iii_{k}}'\|
  \end{equation*}
  for all $x \in V$ and $\kkk \in \Sigma_*$. Since $\delta_{k}\|\fii_{\iii_{k}}'\| \le \eps
  K^{-2}\|\fii_{\jjj_{k-1}\kkk_{k-1}\cdots\jjj_1\kkk_1}'\|\|\fii_{\iii_{k}}'\| \le
  \eps\|\fii_{\iii_{k}\jjj_{k-1}\kkk_{k-1}\cdots\jjj_1\kkk_1}'\|$, there is $\kkk_{k} \in \Sigma_*$
  such that $\fii_{\kkk_{k}}'$ is positive and
  \begin{align*}
    \eps\|\fii_{\iii_{k}\jjj_{k-1}\kkk_{k-1}\cdots\jjj_1\kkk_1}'\|\|\fii_{\kkk_{k}}'\| &< \delta_{k}\|\fii_{\iii_{k}}'\|
    \le \eps\|\fii_{\iii_{k}\jjj_{k-1}\kkk_{k-1}\cdots\jjj_1\kkk_1}'\|\|\fii_{\kkk_{k}|_{|\kkk_{k}|-2}}'\| \\
    &\le K^2\Bigl(\min_{\iii\in\Sigma_2}\|\fii_\iii'\|\Bigr)^{-1}\eps\|\fii_{\iii_{k}\jjj_{k-1}\kkk_{k-1}\cdots\jjj_1\kkk_1}'\|\|\fii_{\kkk_{k}}'\|.
  \end{align*}
  Note that, by \eqref{eq:semiconformal-multi}, $K^{-2}\|\fii_{\iii\kkk\jjj}'\| \le \|\fii_{\iii\jjj}'\|\|\fii_\kkk'\| \le K^4\|\fii_{\iii\kkk\jjj}'\|$ for all $\iii,\jjj,\kkk\in\Sigma_*$. Therefore,
  \begin{equation} \label{eq:not-wsc-assouad-one3}
  \begin{split}
    (K^3C)^{-1}\eps\|\fii_{\iii_{k}\kkk_{k}\jjj_{k-1}\kkk_{k-1}\cdots\jjj_1\kkk_1}'\| &\le \fii_{\iii_{k}\kkk_{k}\jjj_{k-1}\kkk_{k-1}\cdots\jjj_1\kkk_1}(x)-\fii_{\jjj_{k}\kkk_{k}\jjj_{k-1}\kkk_{k-1}\cdots\jjj_1\kkk_1}(x) \\
    &\le K^8C\Bigl(\min_{\iii\in\Sigma_2}\|\fii_\iii'\|\Bigr)^{-1}\eps\|\fii_{\iii_{k}\kkk_{k}\jjj_{k-1}\kkk_{k-1}\cdots\jjj_1\kkk_1}'\|
  \end{split}
  \end{equation}
  for all $x \in V$. We have thus shown the existence of words $\iii_k,\jjj_k,\kkk_k \in \Sigma_*$,
  $k \in \{1,\ldots,n\}$, for which the derivatives $\fii_{\iii_k}'$, $\fii_{\jjj_k}'$, and
  $\fii_{\kkk_k}'$ are positive and \eqref{eq:not-wsc-assouad-one3} holds for all $k \in
  \{1,\ldots,n\}$.

  We will use \eqref{eq:not-wsc-assouad-one3} to define the required points $x_n < x_{n-1} < \cdots < x_1$ in $F$.
  Let $\hhh_k = \iii_n\kkk_n\cdots\iii_k\kkk_k\jjj_{k-1}\kkk_{k-1}\cdots\jjj_1\kkk_1$ and notice
  that, by the chain rule, $\fii_{\hhh_k}'$ is positive for all $k \in \{1,\ldots,n\}$. Therefore,
  by \eqref{eq:semiconformal} and \eqref{eq:not-wsc-assouad-one3}, we have
  \begin{equation*}
  \begin{split}
    \fii_{\hhh_{k}}(x) - \fii_{\hhh_{k+1}}(x) &\le \|\fii_{\iii_n\kkk_n\cdots\iii_{k+1}\kkk_{k+1}}'\|\bigl(\fii_{\iii_{k}\kkk_{k}\jjj_{k-1}\kkk_{k-1}\cdots\jjj_1\kkk_1}(x)-\fii_{\jjj_{k}\kkk_{k}\jjj_{k-1}\kkk_{k-1}\cdots\jjj_1\kkk_1}(x)\bigr) \\
    &\le K^{10}C\Bigl(\min_{\iii\in\Sigma_2}\|\fii_\iii'\|\Bigr)^{-1}\eps\|\fii_{\hhh_k}'\|
  \end{split}
  \end{equation*}
  and, similarly,
  \begin{equation*}
    \fii_{\hhh_{k}}(x) - \fii_{\hhh_{k+1}}(x) \ge (K^{6}C)^{-1}\eps\|\fii_{\hhh_k}'\| > 0
  \end{equation*}
  for all $x \in V$ and $k \in \{1,\ldots,n-1\}$. Recalling \eqref{eq:semiconformal-diam} and the definition of $D \ge 1$ given in \eqref{eq:def-of-D}, we have thus shown that
  \begin{equation} \label{eq:not-wsc-assouad-one4}
    D^{-1}\eps\diam(\fii_{\hhh_k}(F)) \le \fii_{\hhh_k}(x) - \fii_{\hhh_{k+1}}(x) \le D\eps\diam(\fii_{\hhh_k}(F))
  \end{equation}
  for all $x \in V$ and $k \in \{1,\ldots,n-1\}$. Let $y,z \in F$ be such that $\fii_{\hhh_{k+1}}(y)
  - \fii_{\hhh_{k+1}}(z) = \diam(\fii_{\hhh_{k+1}}(F))$. Since $\fii_{\hhh_{k+1}}'$ and
  $\fii_{\hhh_k}'$ are positive, we have $z<y$ and $\fii_{\hhh_k}(z) < \fii_{\hhh_k}(y)$. Therefore,
  by \eqref{eq:not-wsc-assouad-one4}, we have
  \begin{equation*}
    \diam(\fii_{\hhh_k}(F)) \ge \fii_{\hhh_k}(y) - \fii_{\hhh_k}(z) \ge \diam(\fii_{\hhh_{k+1}}(F)) - 2D\eps\diam(\fii_{\hhh_k}(F))
  \end{equation*}
  and
  \begin{equation*}
    \diam(\fii_{\hhh_{k+1}}(F)) \le (1+2D\eps)\diam(\fii_{\hhh_k}(F))
  \end{equation*}
  for all $k \in \{1,\ldots,n-1\}$. Choosing $z,y \in F$ such that $\fii_{\hhh_{k}}(y) -
  \fii_{\hhh_{k}}(z) = \diam(\fii_{\hhh_{k}}(F))$, we similarly see that
  \begin{equation*}
    \diam(\fii_{\hhh_{k+1}}(F)) \ge (1-2D\eps)\diam(\fii_{\hhh_k}(F))
  \end{equation*}
  for all $k \in \{1,\ldots,n-1\}$. By the choice of $\eps>0$, we have thus shown that
  \begin{equation} \label{eq:not-wsc-assouad-one5}
  \begin{split}
    \tfrac12 \diam(\fii_{\hhh_1}(F)) &\le (1-2D\eps)^{n-1}\diam(\fii_{\hhh_1}(F)) \le \diam(\fii_{\hhh_k}(F)) \\
    &\le (1+2D\eps)^{n-1}\diam(\fii_{\hhh_1}(F)) \le 2\diam(\fii_{\hhh_1}(F))
  \end{split}
  \end{equation}
  for all $k \in \{1,\ldots,n\}$. Fix $x_0 \in F$ and define $x_k = \fii_{\hhh_k}(x_0)$ for all $k
  \in \{1,\ldots,n\}$. It follows from \eqref{eq:not-wsc-assouad-one4} that $x_n < x_{n-1} < \cdots
  < x_1$. Letting $x = (x_n+x_1)/2$ and $r=(x_1-x_n)/2$, we have $M_{x,r}(x_n)=-1$ and
  $M_{x,r}(x_1)=1$. Finally, since \eqref{eq:not-wsc-assouad-one4} and
  \eqref{eq:not-wsc-assouad-one5} imply
  \begin{align*}
    x_1-x_n &= \sum_{k=1}^{n-1} \fii_{\hhh_k}(x_0)-\fii_{\hhh_{k+1}}(x_0) \\
    &\ge D^{-1}\eps \sum_{k=1}^{n-1} \diam(\fii_{\hhh_k}(F)) \ge \tfrac12 D^{-1}\eps (n+1) \diam(\fii_{\hhh_1}(F))
  \end{align*}
  and
  \begin{equation*}
    x_k-x_{k+1} \le D\eps\diam(\fii_{\hhh_k}(F)) \le 2D\eps\diam(\fii_{\hhh_1}(F))
  \end{equation*}
  for all $k \in \{1,\ldots,n-1\}$, we see that
  \begin{equation*}
    M_{x,r}(x_k) - M_{x,r}(x_{k+1}) = \frac{2(x_k-x_{k+1})}{x_1-x_n} \le \frac{8D^{2}}{n+1}
  \end{equation*}
  as required.
\end{proof}

\begin{ack}
  AK was supported by the Finnish Center of Excellence in Analysis and Dynamics Research, the
  Finnish Academy of Science and Letters, and the V\"ais\"al\"a Foundation. ST was supported by
  EPSRC Grant DTG/EP/K503162/1, NSERC Grants 2016-03719 and 2014-03154, and the University of
  Waterloo. JA and ST thank the Edinburgh Mathematical Society for their financial support. Finally,
  AK and ST thank the Institut Mittag-Leffler for an excellent working environment, and
  Bal\'azs B\'ar\'any and the anonymous referee for helpful comments.
\end{ack}

\bibliographystyle{abbrv}
\bibliography{Bibliography}


\end{document}